\documentclass[pdftex, oneside, 10pt, letterpaper]{amsart}

\usepackage{verbatim} 
\usepackage{amsthm, amssymb, amsmath}

\usepackage{graphicx}
\usepackage[usenames, dvipsnames]{color}

\usepackage{listings}
\usepackage{upquote} 

\usepackage{microtype}

\input{glyphtounicode} \pdfgentounicode=1

\usepackage{geometry}

\usepackage[pdftex,%
            hyperfigures,%
            hidelinks,%
            pdftitle={The Singular Values of the GUE (Less is More)},%
            pdfdisplaydoctitle=true,%
            breaklinks=false,%
            bookmarks=true,%
            unicode=true,
            pdffitwindow=true,%
            pdfauthor={Alan Edelman and Michael LaCroix},%
            pdfsubject={Random Matrices},%
            pdfkeywords={random matrices, GUE, anti-GUE, LUE, singular values,
              condition number, semicircle law, quarter-circle law},%
]{hyperref} 

\providecommand{\coloneqq}{\mathrel{\mathop:}=}

\newsavebox{\gwbox}\newlength{\gwidth}
\newcommand{\wrapminipage}[2][]{%
  \sbox{\gwbox}{#2}%
  \settowidth{\gwidth}{\usebox{\gwbox}}%
  \begin{minipage}[#1]{\gwidth}\usebox\gwbox\end{minipage}}
\newcommand{\gminipage}[2][]{%
  \wrapminipage{\includegraphics[#1]{#2}}}

\newcommand{\defdelim}[1]{\ifx\relax#1\relax\def\lsize{\left}\def\rsize{\right}\else\def\lsize{#1}\def\rsize{#1}\fi}
\newcommand{\floor}[2][]{\defdelim{#1}\lsize\lfloor#2\rsize\rfloor}
\newcommand{\ceil}[2][]{\defdelim{#1}\lsize\lceil#2\rsize\rceil}
\newcommand{\abs}[2][]{\defdelim{#1}\lsize\lvert#2\rsize\rvert}

\def\<#1>{\left\langle{#1}\right\rangle}
 
\DeclareMathOperator{\tr}{tr}
\newcommand{\E}[2][]{\operatorname{E}_{#1}\left[#2\right]}
\newcommand{\Var}{\operatorname{Var}}
\newcommand{\differentials}[3]{\;\prod_{i=#2}^{#3} \mathrm{d}{#1}_i}

\theoremstyle{plain}
\newtheorem{thm}{Theorem}
\newtheorem{lemma}{Lemma}
\newtheorem{cor}{Corollary}

\theoremstyle{definition}
\newtheorem{exam}{Example}

\theoremstyle{remark}
\newtheorem*{rem}{Remark}

\newcommand{\GOE}{\ensuremath{\mathrm{GOE}}}
\newcommand{\GSE}{\ensuremath{\mathrm{GSE}}}
\newcommand{\GBE}{\ensuremath{\mathrm{G}\beta\mathrm{E}}}
\newcommand{\CUE}{\ensuremath{\mathrm{CUE}}}
\newcommand{\GUE}{\ensuremath{\mathrm{GUE}}}
\newcommand{\LUE}{\ensuremath{\mathrm{LUE}}}
\newcommand{\LOE}{\ensuremath{\mathrm{LOE}}}

\newcommand{\herm}{\mathrm{H}}

\title{The Singular Values of the \GUE{} (Less is More)}

\author{Alan Edelman}
\address[Alan Edelman]{Department of Mathematics, Massachusetts Institute of
  Technology}
\email{edelman@mit.edu}

\author{Michael La$\,$Croix}

\address[Michael La$\,$Croix]{Department of Mathematics, Massachusetts Institute of Technology}
\email{malacroi@mit.edu}


\begin{document}
\begin{abstract}
  Some properties that nominally involve the eigenvalues of Gaussian
  Unitary Ensemble (\GUE{}) can instead be phrased in terms of
  singular values.  By discarding the signs of the eigenvalues, we
  gain access to a surprising decomposition: the singular values of
  the \GUE{} are distributed as the union of the singular values of
  two independent ensembles of Laguerre type.  
  This independence is remarkable given the well known phenomenon of
  eigenvalue repulsion.

  The structure of this decomposition reveals that several existing
  observations about large $n$ limits of the \GUE{} are in fact
  manifestations of phenomena that are already present for finite
  random matrices.  We relate the semicircle law to the quarter-circle
  law by connecting Hermite polynomials to generalized Laguerre
  polynomials with parameter $\pm1/2$.  Similarly, we write the
  absolute value of the determinant of the $n\times{}n$ \GUE{} as a
  product $n$ independent random variables to gain new insight into
  its asymptotic log-normality.  The decomposition also provides a
  description of the distribution of the smallest singular
  value of the \GUE{}, which in turn permits the study of the leading
  order behavior of the condition number of \GUE{} matrices.

  The study is motivated by questions involving the enumeration of
  orientable maps, and is related to questions involving powers of
  complex Ginibre matrices.  The inescapable conclusion of this work
  is that the singular values of the GUE play an unpredictably
  important role that had gone unnoticed for decades even though, in
  hindsight, so many clues had been around.
\end{abstract}

\maketitle

\section{Introduction}

This paper highlights some surprising interrelationships between
problems that involve singular values of \GUE{} random matrices.  By
discarding the signs of eigenvalues, we gain access to additional
structure, since despite the pairwise repulsion of its eigenvalues,
the singular values of the \GUE{} can be decomposed as the union of
two independent sets.  The decomposition is equivalent to a result of
Jackson and Visentin \cite{JV-Eulerian} from enumerative
combinatorics, and was previously reported by Forrester in
\cite[Sec.~2.2]{Forrester-Evenness}.  Our contribution is to consider
the decomposition as a complete result about singular values instead
of a specialized result about eigenvalues, and to note that this
single decomposition underlies several diverse phenomena.  From this
perspective, we can translate results about asymptotically large
matrices to the finite setting, and we can capitalize on the
independence to describe the determinant and extreme singular values
of the \GUE{}.

Several results, that we find individually surprising, are in fact hidden
facets of the same phenomenon.  Our aim is to expose these surprises
and the interconnections between them.
\begin{enumerate}
\item It is possible to partition the singular values of the \GUE{}
  into two statistically independent sets (stated in
  \cite{Forrester-Evenness} in terms of eigenvalues).  This stands
  in striking contrast, almost in contradiction with, the familiar
  fact that eigenvalues repel.
\item The logarithm of the absolute value of the determinant of the
  \GUE{} can be written as a sum of independent random variables
  (speculated as impossible by Tao and Vu in \cite{TV}).
\item Matrices of nominal half-integer size play a key role.
\item The decomposition is equivalent to a result from enumerative
  combinatorics that relates the cardinalities of two classes of
  orientable maps on surface of positive genus (\cite{JV-Eulerian}).
\item A bi-diagonal model for singular values gives all the moments of
  the \GUE{} determinant.
\item The bulk-scaling limit of the \GUE{} behaves as a superposition
  of two hard edges.  The first author has long since argued for the
  relatively obvious importance of the singular value view for
  Laguerre (Wishart) ensembles, and the less well known, but easy to
  recognize, generalized singular value view for Jacobi (MANOVA)
  ensembles (see the first author's course notes for course 18.337 at
  MIT).  The importance of a singular value view for the \GUE,
  however, is far more astonishing.
\end{enumerate}

Our approach is analogous to replacing a semicircle with a pair of
quarter-circles.  These curves occur as famous limiting distributions.
In particular, Wigner's semicircle law is the limiting distribution of
the eigenvalues of the (\GUE{}).  The Marchenko-Pastur distribution
similarly describes the limiting distribution for the singular values
of large rectangular random matrices.  In particular, Laguerre
ensemble singular values satisfy the quarter-circle law.  There is an
obvious geometric relationship between these distributions; a
semicircle is the union of two quarter-circles
(Figure~\ref{fig:semicircles}--top).  The semicircle and
quarter-circles also have a less obvious relationship: the semicircle
is symmetric about the $y$-axis, and its restriction to the first
quadrant is the average of two quarter-circles.  This second
relationship generalizes to matrices of finite size
(Figure~\ref{fig:semicircles}--right), with the quarter-circles
replaced by the distributions of singular values of rectangular
matrices of nominal half-integer size
(Figure~\ref{fig:semicircles}--bottom).  Variations of this second
relationship form the basis for this paper.

\begin{figure}
\begin{minipage}{.725\textwidth}
  \hspace*{\fill}\includegraphics[width=\textwidth]{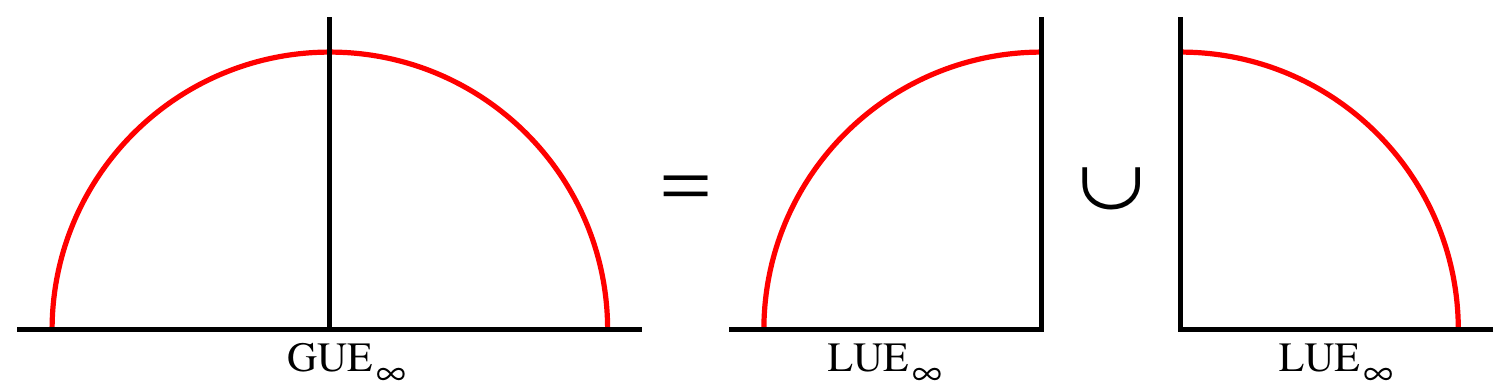}\hspace{\fill}\mbox{}\\[2ex]

  \hspace*{\fill}\includegraphics[width=\textwidth]{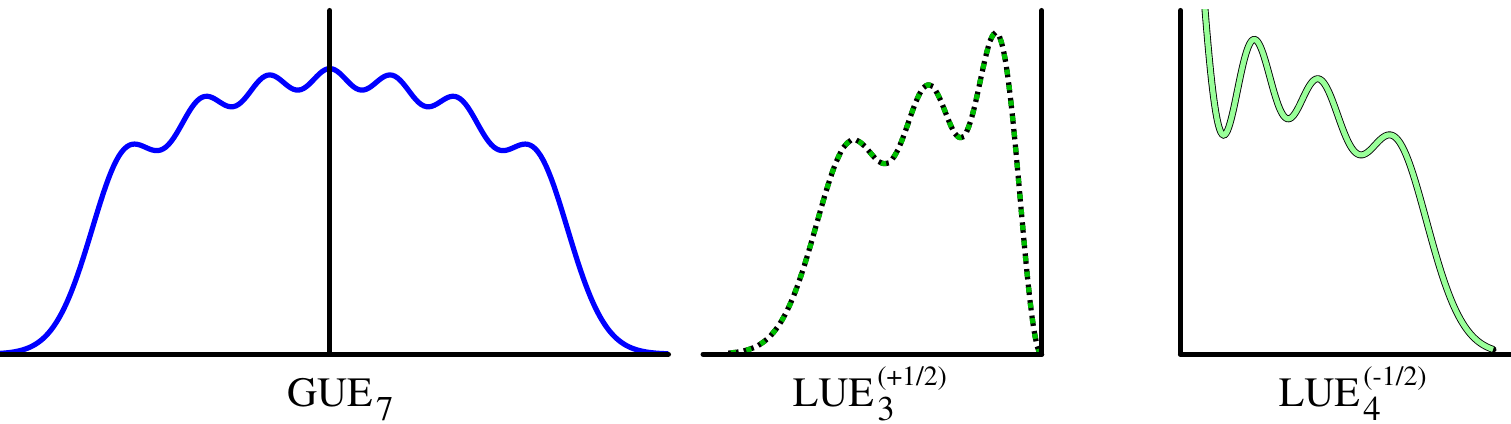}\hspace{\fill}\mbox{}
\end{minipage}\hspace{\fill}%
\begin{minipage}{.2083\textwidth}
  \includegraphics[width=\textwidth]{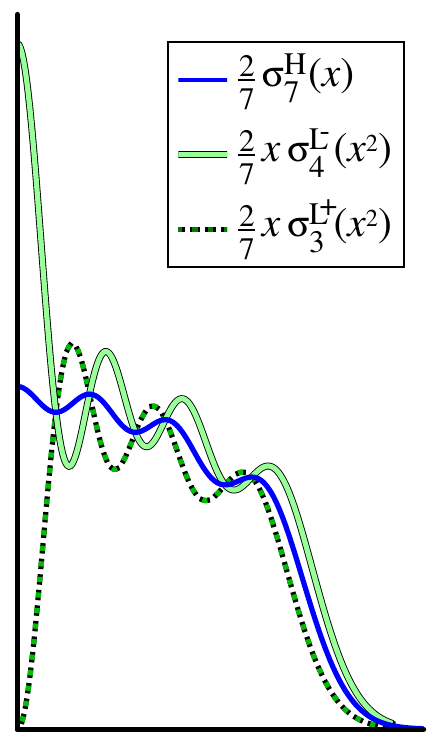}
\end{minipage}
\caption{A semicircle describes the limiting density of the
  eigenvalues of the \GUE{}.  It decomposes as two quarter-circles
  (\textcolor{BrickRed}{top}), related to the limiting densities of
  the singular values of rectangular matrices.  For finite matrices,
  the density of a random eigenvalue of the \GUE{}
  (\textcolor{blue}{bottom left}) is still described by a weighted
  average (right) of the densities of
  of the positive square roots of the eigenvalues from two \LUE{}s
  (\textcolor{OliveGreen}{bottom} \textcolor{OliveGreen}{center}).
}\label{fig:semicircles}
\end{figure}

Much of random matrix theory involves the behavior of eigenvalues of
asymptotically large matrices.  It is not always clear how such
phenomena correspond to finite matrices.  In this paper, we connect
the infinite to the finite by phrasing phenomena in terms of singular
values.  For Hermitian matrices, this amounts to considering the
magnitudes of eigenvalues and discarding their signs.  One might
assume that discarding signs limits the scope of possible conclusions,
but in practice several problems that are nominally about eigenvalues
are better analyzed in terms of singular values.  One could even argue
that existing results about the extreme eigenvalues of Laguerre and
Jacobi ensembles are elegant precisely because they are essentially
about singular values.
  
The change of setting becomes advantageous when we observe that the
singular values of the $n\times n$ \GUE{} exhibit an unexpected
decomposition: Theorem~\ref{thm:mainresult} shows that they are
distributed identically to the union of the distinct non-zero singular
values of two independent anti-\GUE{} ensembles (an anti-\GUE{} matrix
consists of purely imaginary Gaussian entries that are independently
distributed subject to skew-symmetry) one of order $n$, the other of
order $n+1$.  An equivalent result was previously observed by
Forrester in \cite[Sec.~2.2]{Forrester-Evenness} where it was stated
explicitly for the case that $n$ is even.  Since the eigenvalues of
the \GUE{} are readily seen to be pairwise dependent, the existence of
such a decomposition is itself somewhat surprising.

The decomposition allows us to analyze several statistics of the
\GUE{}, including the physically significant gap probability, in terms
of the anti-\GUE{}.  Ironically, most of the relevant facts about the
anti-\GUE{} can be found in Mehta's physically motivated text,
\cite[Ch. 13]{Mehta}, where his description is asserts that such
matrices have ``no immediate physical interest''.  After a change of
variables, the positive eigenvalues of the anti-\GUE{} are seen to
have distributions of Laguerre-type (Section~\ref{sec:antiGUE}),
corresponding to complex matrices with a half-integral dimension and
Laguerre parameter $\pm\frac{1}{2}$ (this is the $\beta=2$ case of a
more general analysis presented by Dumitriu and Forrester in
\cite{DuFo}).  It is thus possible to draw conclusions about the
\GUE{} from an understanding of corresponding facts about Laguerre
ensembles.  Physically significant existing results about level
densities, the absolute value of the determinant, the distributions of
the largest singular value, and the bulk-scaling limit can all be
analyzed using this framework.

As an unexpected consequence, we obtain the square of the determinant
of the $n\times n$ \GUE{} as a product of independent $\chi^2$ random
variables (Theorem~\ref{thm:absdetGUE}).  This is a direct analogue to
the result of Goodman for Wishart matrices \cite{Goodman}, and
precisely the form that Tao and Vu speculated did not exist
when discussing the log-normality of the absolute value of the
determinant of the \GUE{} in \cite{TV}.

In addition to providing a common framework for understanding existing
results about the \GUE{}, the decomposition permits a study of the
distribution of the smallest singular value of a matrix from the
ensemble.  This quantity may initially appear somewhat unnatural, but
for some applications it is an appropriate analog for the smallest
eigenvalue of Laguerre and Jacobi ensembles, in some ways behaving as
though governed by the existence of a virtual hard-edge.  The
distribution of the smallest singular value is also closely related to
the distribution of conditions numbers, and has implications for the
analysis of numerical stability of operations involving random
matrices.

The decomposition was first identified by the authors as part of an
attempt to find a combinatorial derivation for a functional identity,
given by Jackson and Visentin in \cite{JV-Characters}, between
generating series for two classes of orientable maps.  Physical
implications of their identity involve matrix models of 2-dimensional
gravity, and are discussed in \cite{JPV}.  They later generalized the
identity, in \cite{JV-Eulerian}, to a stronger form that is essentially
equivalent to the existence of our decomposition.  Their generating
series are effectively cumulant generating series for suitably scaled
ensembles of matrix eigenvalues, but Jackson and Visentin appear to
have been unaware of the random matrix interpretation of one of the
series, possibly because its direct interpretation involves a
half-integer evaluation of a parameter that nominally represents one
of the dimensions of a rectangular matrix of complex Gaussians.  While
their work required subtle manipulation of characters of the symmetric
group, we believe that the present proof is elementary and
enlightening from the perspective of random matrix theory, although a
combinatorial interpretation still remains elusive.

It should be noted that while the decomposition discussed here has
many superficial parallels with the ideas of superposition and
decimation superposition explored by Forrester and Rains
(\cite{FoRa-OUS, FoRa-Decimation}), the concepts are distinct,
although it is not difficult to imagine a more general setting in
which both their result and ours exist as special cases.

\subsection*{Outline}
The remainder of the paper has the following structure:
\begin{itemize}
\item Section~\ref{sec:definitions} describes the matrix
  ensembles we need to formulate the decomposition.
\item Section~\ref{sec:leveldensities} uses the level density of the
  \GUE{} as a warm-up exercise.
\item Section~\ref{sec:decomposition} demonstrates the decomposition.
  We also describes its equivalence to an identity of Jackson and
  Visentin, and discuses how the decomposition can be observed
  experimentally.
\item Section~\ref{sec:applications} applies the decomposition to
  provide a unified explanation to existing results.
\item Section~\ref{sec:Ginibre} relates the decomposition to
  properties of the complex Ginibre ensemble, and draws parallels to
  an earlier investigation by Rains of powers of compact Lie groups
  \cite{Rains-Powers, Rains-power-images}.
\item Finally, in Section~\ref{sec:future} we discuss some related questions
  for future work.
\end{itemize}

\section{The Ensembles}\label{sec:definitions}

\subsection*{Gaussian Unitary Ensembles}\label{sec:GUE}

The Gaussian Unitary Ensemble of order $n$, ($\GUE_{n}$), consists of
$n\times{}n$ Hermitian matrices invariant after conjugation by any
unitary matrix, and with entries that are normal, and independently
distributed, subject to Hermitian symmetry.  The ensemble is
completely defined by specifying the variance of the diagonal entries,
and we choose a normalization with diagonal entries standard normal.
As a consequence, the real and imaginary parts of the off-diagonal
entries are independently normal with mean $0$ and variance $\frac12$.
The ensemble can be sampled as $A=\frac{1}{2}(G+G^{\herm})$, where the
real and imaginary parts of the entries of the $n\times{}n$ matrix $G$
are independently standard normal, and $G^{\herm}$ denotes the
Hermitian conjugate of $G$.

\begin{rem}
  It is also common to work with a normalization where real and
  imaginary parts of the off-diagonal entries are standard normal, as
  in \cite{Mehta, MeNo}, or where the variance depends on $n$ (when
  the primary concern is taking large-$n$ limits).  Our choice is
  motivated by combinatorial considerations from the map enumeration
  setting studied by Jackson and Visentin (\cite{JV-Eulerian}), and
  provides the property that for every partition
  $\theta=(\theta_1,\theta_2,\dotsc,\theta_k)$, the moment
  $m_\theta(n)=\E[\GUE_n]{\prod_{i=1}^k\tr(M^{\theta_i})}$ is a
  polynomial in $n$ with non-negative integer coefficients depending
  only on $\theta$.  A convenient consequence of this normalization is
  that $\E[\GUE_n]{\det(M^{2k})}$ is a product of odd integers for
  every $n$ and $k$ (see Theorem~\ref{thm:absdetGUE}).
\end{rem}

An element of the \GUE{} has real eigenvalues, so the distribution on
the matrices induces a distribution on $n$-tuples of eigenvalues.  The
joint density function for this distribution on $\mathbb{R}^n$, is
\begin{equation}\label{eqn:GUEdensity}
  p_n^{\mathrm{H}}(x_1,x_2,\dotsc,x_n) =
  c_{n}^{\mathrm{H}}{\prod_{1\leq i<j\leq{}n}\abs{x_i-x_j}^2}
  \;\exp\Big(-\frac12\sum_{i=1}^n x_i^2\Big)
  \differentials{x}{1}{n}
\end{equation}
where $c_{n}^{\mathrm{H}}$ is such that the density defines a
probability measure.  A thorough discussion of the \GUE{} is given by
Mehta in \cite{Mehta}, though with a different choice of
normalization.  It is convenient to consider the density as consisting
of two factors: the Vandermonde squared factor,
$\prod_{1\leq{}i<j\leq{}n}\abs{x_i-x_j}^2$, occurs because the
ensemble is unitarily invariant, while the second factor,
$\exp\big(-\frac12\sum_{i=1}^n x_i^2\big)$, is associated to the
Hermite weight in the study of orthogonal polynomials (explaining the
use of `$\mathrm{H}$' in our notation), and occurs because the density
of a matrix $M$ is proportional to $\exp\big(-\frac12\tr(M^2)\big)$.

\begin{rem}
  It is also common to consider the \GUE{} in terms of a density on
  sets of eigenvalues, and thus use a density that is supported only
  on $x_1\leq{}x_2\leq\dotsm\leq{}x_n$.  For the present purposes, we
  prefer to have a density function that is invariant under
  permutation of its arguments, and so we consider a density on
  $n$-tuples constructed by randomly permuting the eigenvalues.  The
  two approaches are not substantially different, but would manifest
  as a factor of $n!$ if $c_n^{\mathrm{H}}$ were to be stated
  explicitly.  In Section~\ref{sec:decomposition} we will consider an
  alternate density on $n$-tuples that induces the same density on
  sets.
\end{rem}

\subsection*{Laguerre Unitary Ensembles}\label{sec:Laguerre}

The Laguerre Unitary Ensembles (\LUE) are a two-parameter family of
distributions on positive definite Hermitian matrices.  The parameter
$n$ corresponds to the order of the matrix, while the parameter $a$
determines the shape of the distribution.  In contrast to the \GUE{}
which traditionally found applications arising in physics, the \LUE{}
are more closely associated with statistics where the relevant
matrices are often referred to as Wishart matrices.  Many statistical
applications of the \LOE{}, an analogous ensemble based on real
instead of complex matrix entries, can be found in \cite{Muirhead},
and much of the commentary there applies to the \LUE{} with minor
modifications.  There are two related models of the \LUE{}: one model
applies when $n+a$ is a positive integer, while a second model applies
when $a+1$ is a positive real number.

When $n+a$ is a positive integer, the ensemble can be sampled as
$W=AA^\herm$, where $A$ is an $n\times{}(n+a)$ matrix of independent
complex Gaussian entries.  The spectrum of $A$ is completely
determined by $A$, and it is often more natural to work with the
singular values of $A$ than with the eigenvalues of $W$.  By
convention we will consider centered matrix entries, with equally
distributed real and imaginary parts chosen such that
$\E{A_{ij}\overline{A_{ij}}}=2$, so that the real and imaginary parts
of each entry are independent standard normal.  With this
normalization, the joint density for the eigenvalues of the \LUE{} on
$[0,\infty)^n$, is
\begin{equation}\label{eqn:laguerredensity}
p_{n,a}^{\mathrm{L}}(x_1,x_2,\dotsc,x_n)=
c_{n,a}^{\mathrm{L}}  \prod_{i=1}^{n}{x_i^{a}
  \prod_{1\leq i<j\leq n}\abs{x_i-x_j}^2}\,
  \exp\big(-\frac{1}{2}\sum_{i=1}^n x_i\big)
  \differentials{x}{1}{n}.
\end{equation}
In fact \eqref{eqn:laguerredensity} continues to define a probability
density for non-integral $a>-1$, and the densities are
realized when $A$ is bi-diagonal with its non-zero entries
independently $\chi$-distributed according to
\begin{equation}\label{eqn:bi-diagonalform}
  A\sim\begin{pmatrix} \chi_{2(n+a)} \\ \chi_{2(n-1)} &
  \chi_{2(n+a-1)} \\ & \chi_{2(n-2)} & \chi_{2(n+a-2)} \\ 
  & \makebox[0pt]{\hspace*{5em}$\ddots$} & \makebox[0pt]{\hspace*{4em}$\ddots$} \\ 
  & & \chi_2 & \chi_{2(a+1)}
  \end{pmatrix}.
\end{equation}
The correctness of this model for integer values of $a$ is verified by
considering the effect of applying Householder reflections to a matrix
$A$ of complex Gaussians, and can be seen to extend to non-integral
$a$ via the fact that the moments of \eqref{eqn:laguerredensity} must
depend polynomially on $a$.  A complete derivation of the bi-diagonal
model for Laguerre ensembles is given in a more general setting in
\cite{DuEd}.  In the present paper, we will be primarily interested in
Laguerre ensembles for which $a=\pm\frac12$ and their relationship to
anti-\GUE{} matrices of even and odd order, although ensembles
corresponding to arbitrary values of $a$ are closely related to the
combinatorics in \cite{JV-Eulerian} that motivated the present study.

\begin{rem}
  As with the \GUE{}, moments of the \LUE{} can be interpreted
  combinatorially.  Taking $m=n+a$, the moments
  $m_\theta(n,m)=\E[\LUE_n^{(a)}]{\prod_{i=1}^k\tr(M^{\theta_i})}$ are
  each polynomials in $m$ and $n$ with non-negative integer
  coefficients and are symmetric in $m$ and $n$.  These coefficients
  are related to the enumeration of hypermaps and associated with the
  generating series discussed in \cite{JV-Eulerian}, though a direct
  interpretation of the combinatorial results there requires the
  alternate normalization $\E{A_{ij}\overline{A_{ij}}}=1$.
\end{rem}

\subsection*{Anti-GUE}\label{sec:antiGUE}

The anti-\GUE{} consists of anti-symmetric Hermitian matrices with
independent (subject to anti-symmetric) normal entries.  Such matrices
were identified by Mehta as having a particularly elegant theory, with
no immediate applications to physics \cite[Ch.~13]{Mehta}.  Every such
matrix is of the form $M=iK$, where $K$ is a real skew-symmetric
matrix.  Such a matrix is unitarily diagonalizable, so its singular
values are the absolute values of its eigenvalues.  Since the
characteristic polynomial of $K$ has real coefficients, its
eigenvalues occur in complex conjugate pairs, and it follows that the
eigenvalues of $M$ occur in plus/minus pairs, so each non-zero
singular values occurs with even multiplicity.  If $M$ is $N\times{}N$
for $N=2n+r$ with $r\in\{0,1\}$, then except on a set of measure zero,
$M$ has $n$ distinct non-zero singular values, which we can denote by
$\theta_1,\theta_2,\dotsc,\theta_n$.  When the imaginary parts of the
entries of $M$ are distributed as independent standard Gaussians (up
to Hermitian symmetry), the joint probability density function for the
distinct singular values of $M$ (in this case also the positive
eigenvalues), supported on $[0,\infty)^n$, is given by
\begin{equation}\label{eqn:skewdensity}
  p^{\mathrm{aG}}_N(\theta_1,\theta_2,\dotsc,\theta_n)=
  c^{\mathrm{aG}}_{N}\prod_{j=1}^n\theta_j^{2r}\,\prod_{1\leq j<k\leq n}
  \left(\theta_j^2-\theta_k^2\right)^2
  \,\exp\big(-\frac{1}{2}\sum_{j=1}^n\theta_j^2\big)
    \differentials{\theta}{1}{n},
\end{equation}
an expression that combines the two cases described in
\cite[Section~3.4]{Mehta} or \cite[Ex~1.3~q.5]{Forrester-Log-gases}
after accounting for the differing choice of normalization.  Key to
the existence of the decomposition in Theorem~\ref{thm:mainresult} is
that the final factor,
$\exp\big(-\frac{1}{2}\sum_{j=1}^n\theta_j^2\big)$, which is common to
both this density and the \GUE{} density, is a symmetric product of
even functions.

For completeness, we will outline how this density can be derived from
the Laguerre density, \eqref{eqn:laguerredensity}, by establishing the
existence of a bi-diagonal model for the singular values.  This
follows closely one of the approaches used by Dumitriu and Forrester
in \cite{DuFo}, where several other derivations are also presented.
By applying a sequence of orthogonal Householder transformation to
$M$, it is seen to have the same eigenvalue distribution as the
tri-diagonal anti-symmetric matrix
\[
\mathrm{i}\begin{pmatrix}
  0 & \chi_{N-1} \\
  -\chi_{N-1} & 0 & \chi_{N-2} \\
  & -\chi_{N-2} & 0 & \chi_{N-3} \\
  & & -\chi_{N-3} & 0 & \ddots \\
  & & & \ddots & \ddots & \chi_1 \\
  &&&& -\chi_1 & 0
\end{pmatrix} ,
\]
which by simultaneously permuting rows and columns is orthogonally
similar to a matrix of the form $\mathrm{i}\begin{pmatrix} 0 & A \\
  -A^T & 0\end{pmatrix}$, where depending on the parity of $N$,
\[
    A_{N_{\text{odd}}} \sim
    \begin{pmatrix}
      \chi_{N-1} & \chi_{N-2} \\ & \chi_{N-3} & \chi_{N-4} \\ & &
      \ddots & \ddots \\ & & & \chi_{2} & \chi_{1}
    \end{pmatrix}
  \quad \text{or} \quad
    A_{N_{\text{even}}} \sim
    \begin{pmatrix}
      \chi_{N-1} & \chi_{N-2} \\
      & \chi_{N-3} & \chi_{N-4} \\
      & & \ddots & \ddots \\
      & & & \chi_3 & \chi_2 \\
      & & & & \chi_{1} 
    \end{pmatrix}.
\]
Despite the differing form for even and odd $N$, for many purposes
these bi-diagonal should be considered as comprising a one-parameter
family.  Their moments, for example, can be seen to linked, and to
depend polynomially on $N$.  Figure~\ref{fig:antiGUEstaircase}
emphasizes the uniformity by illustrating how each bi-diagonal matrix
is obtained from one of lower order by adding a single additional
non-zero matrix element.  Notice that when $N$ is odd, the matrix
$A_N$ is not square.

\begin{figure}
\hspace*{\fill}\includegraphics[width=.9\textwidth]{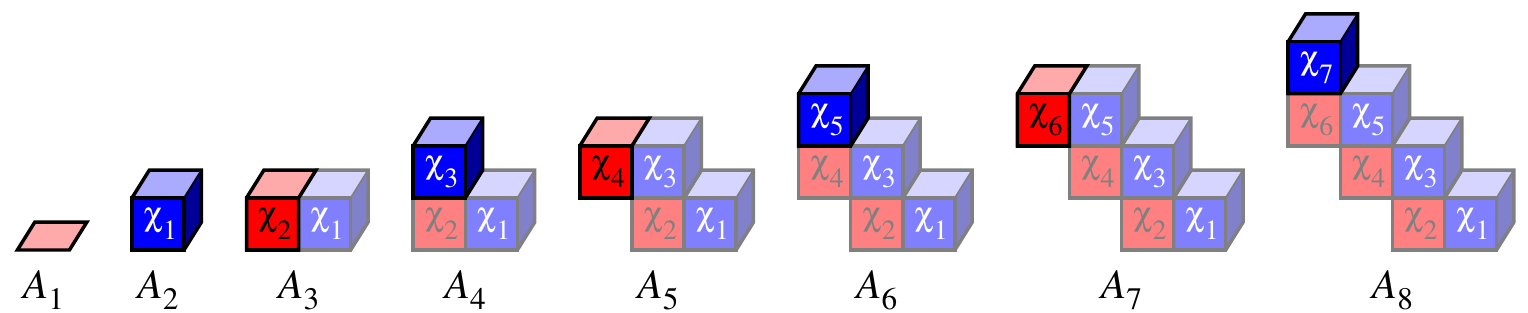}\hspace{\fill}\mbox{}
\caption{The bi-diagonal models for the positive eigenvalues of the
  anti-\GUE{} form a one-parameter family, where each is obtained from
  the previous by adding a single non-zero matrix
  element.}\label{fig:antiGUEstaircase}
\end{figure}

When $N$ is even, the matrix $A_{N_\text{even}}$ is the transpose of
the Laguerre form from \eqref{eqn:bi-diagonalform}, with $a=-\frac12$.
For odd values of $N$, the singular values of $A_{N_\text{odd}}$ can
also be seen to be Laguerre distributed, in this case with
$a=\frac12$, by noting that
\[
A_{N_{\text{odd}}}\sim\begin{pmatrix} \chi_{N-1} & \chi_{N-2} \\ &
  \chi_{N-3} & \chi_{N-4} \\ & & \ddots & \ddots \\ & & & \chi_{2} &
  \chi_{1}
\end{pmatrix}
\quad\text{and}\quad B_{N_\text{odd}}\sim\begin{pmatrix}
  \chi_{N} & \chi_{N-3} \\
  & \chi_{N-2} & \chi_{N-5} \\
  & & \ddots & \ddots \\
  &&& \chi_5 & \chi_2 \\
  &&& & \chi_3
\end{pmatrix}
\]
have identically distributed singular values.  Dumitriu and Forrester
\cite[Claim~6.5]{DuFo} demonstrated this equivalence by noting that
$B_{N_\text{odd}}$ describes the distribution of the Cholesky factor
of $A^\mathrm{T}A$.  The following lemma can be used to establish the
same claim while working directly with $A_{N_\text{odd}}$ and
$B_{N_\text{odd}}$, potentially avoiding numerical pitfalls
associated with constructing $A^\mathrm{T}A$.

\begin{lemma}\label{lem:chilem}
  If $A=\begin{pmatrix}W&0\\X&Y\end{pmatrix}$ has independent entries
  with $W\sim\chi_{r+s}$, $X\sim\chi_r$, and $Y\sim\chi_s$, and $Q$ is
  the reflection matrix
  $Q=\dfrac{1}{\sqrt{X^2+Y^2}}\begin{pmatrix}X&Y\\Y&-X\end{pmatrix}$,
  then $AQ=\begin{pmatrix}T&U\\V&0\end{pmatrix}$ has independent
  entries distributed as $T\sim\chi_r$, $U\sim\chi_s$, and
  $V\sim\chi_{r+s}$.
\end{lemma}

\begin{proof}
  This is equivalent to the more familiar fact that if $W^2$, $X^2$,
  and $Y^2$ are independent with
  $(W^2,X^2,Y^2)\sim(\chi^2_{r+s},\chi^2_{r},\chi^2_{s})$, then
  $X^2+Y^2$, $\frac{W^2X^2}{X^2+Y^2}$, and $\frac{W^2Y^2}{X^2+Y^2}$
  are also independent and distributed as
  $\left(X^2+Y^2,\frac{W^2X^2}{X^2+Y^2},\frac{W^2Y^2}{X^2+Y^2}\right)\sim(\chi^2_{r+s},\chi^2_{r},\chi^2_{s})$.
  This is established by a change of variables in appropriate
  joint probability density functions.
\end{proof}

By iteratively applying the lemma, a matrix distributed as
$A_{N_{\text{odd}}}$ can be orthogonally transformed into one
distributed as $[B_{N_\text{odd}}\mid 0]$ via a sequence of orthogonal
matrices that act on two columns at a time.  Subsequently dropping the
column of zeros does not alter the singular values.  In particular,
the lemma gives a constructive method for sampling $B_{N_\text{odd}}$
from a sample of $A_{N_\text{odd}}$.  Figure~\ref{fig:staircaseequiv}
illustrates the equivalence schematically for $N=7$.

\begin{figure}[!htbp]
\includegraphics[width=\textwidth]{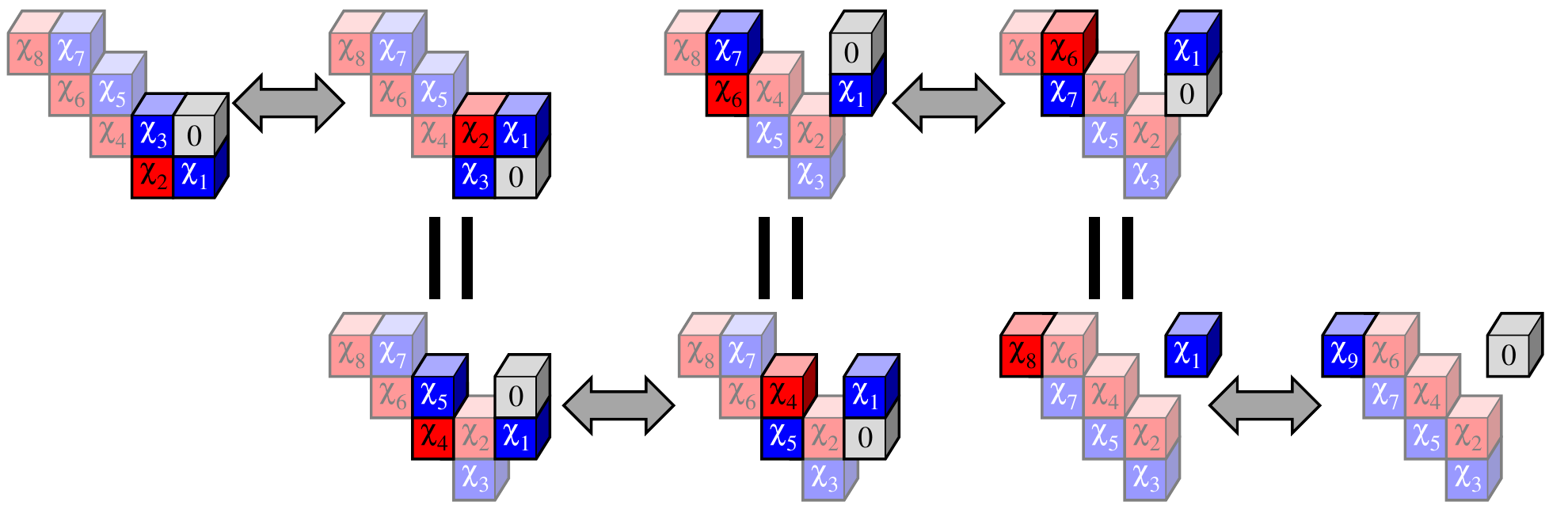}
\caption{Four orthogonal transformations (gray arrows)
  act on two columns at a time to transform a matrix distributed as
  $A_7$ into one distributed as $B_7$.%
}\label{fig:staircaseequiv}
\end{figure}

\begin{rem}
  Heuristically, the equivalence between the singular value
  distributions of $A_{N_{\text{odd}}}$ and $B_{N}$ can be anticipated by
  considering the effect of applying Householder reflections to
  bi-diagonalize a hypothetical complex random matrix with fractional
  size, namely $\frac{n-1}{2}\times{}\frac{n}{2}$.  Beginning the
  process by reducing the first column and then alternating between
  rows and columns produces the first distribution, while starting
  with the first row produces the second distribution.
\end{rem}

In both the cases of even $N$ and odd $N$, the singular values of an
anti-\GUE{} matrix are the singular values of a bi-diagonal matrix of
Laguerre type (Figure~\ref{fig:staircase-Laguerre}), and the
probability density function \eqref{eqn:skewdensity} follows from
\eqref{eqn:laguerredensity} after a change of variable, taking
$\theta_j^2=x_j$ and thus
$2\,\mathrm{d}\theta_j=x^{-1/2}\,\mathrm{d}x_j$, with additional
factors of $2$ absorbed into $c_n^{\mathrm{aG}}$.

\begin{figure}
\hspace*{\fill}\includegraphics[width=.9\textwidth]{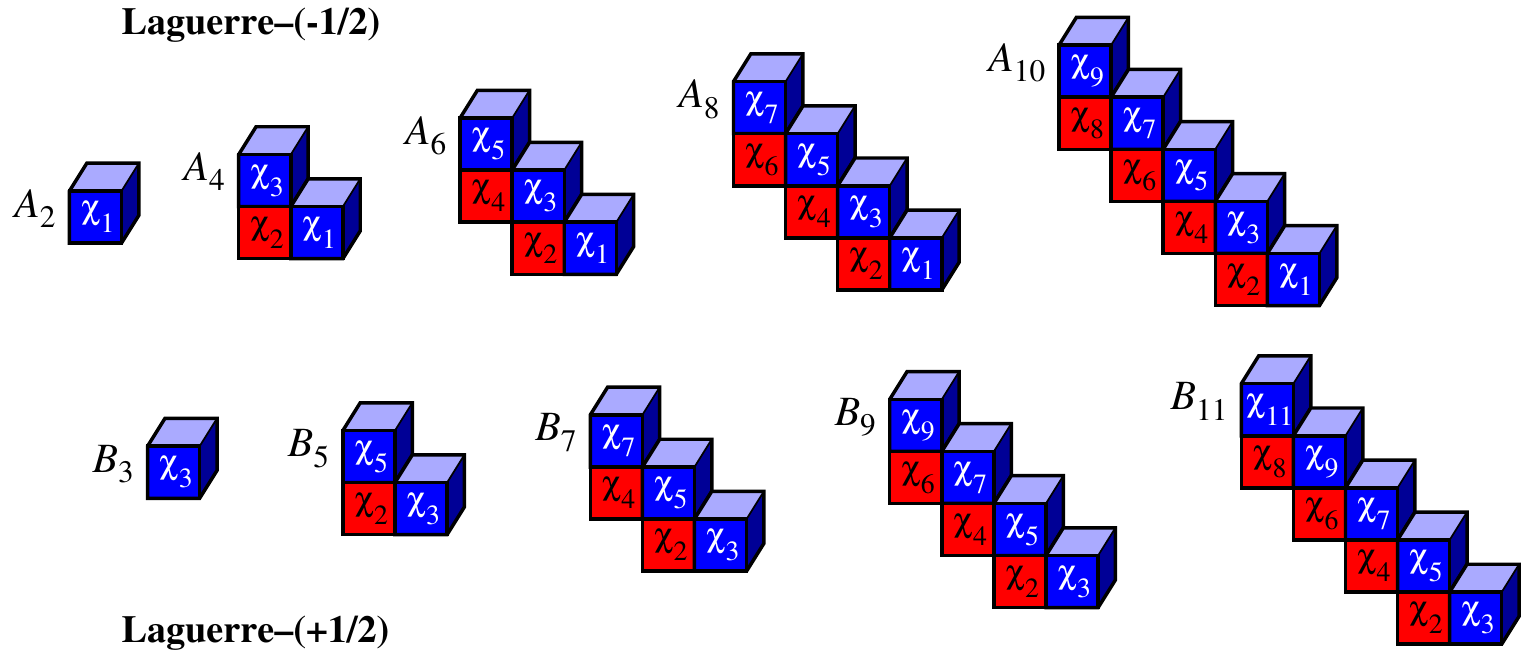}\hspace{\fill}\mbox{}
\caption{The positive eigenvalues of even and odd order anti-\GUE{}
  matrices are modeled by two different families of square bi-diagonal
  matrices of Laguerre type.}\label{fig:staircase-Laguerre}
\end{figure}

\begin{rem}
  It can also be advantageous to view the equivalence between the
  anti-\GUE{} and Laguerre ensembles from the opposite perspective.
  In particular, the relationship formalizes a sense in which the
  ensembles $\{\LUE_{k}^{(+1/2)}\}_{k=1}^{\infty}$ and
  $\{\LUE_{k}^{(-1/2)}\}_{k=1}^{\infty}$ are naturally part of a
  single one-parameter family.  In particular, the moments of
  $\LUE_n^{(1/2)}$ and $\LUE_n^{(-1/2)}$ share the same polynomial
  dependence on $n$, with each evaluated at half-integers relative to
  the other.  This matches our intuition that for the purpose of
  considering singular values, the dimensions of a rectangular matrix
  should be interchangeable, so that both $\LUE_{3}^{(+1/2)}$ and a
  hypothetical $\LUE_{3.5}^{(-1/2)}$ should involve the singular
  values of a nominal $3\times3.5$ matrix.
\end{rem}

\goodbreak

\section{Warm-up:  the level densities of the \GUE{} and the
  Semicircle Law}\label{sec:leveldensities}

Before proceeding to the general setting, we examine more closely the
motivating problem.  How is the semicircle from the \GUE{} related to
the quarter-circles describing singular values of bi-diagonal matrices
of Laguerre type?  What is the analogous relationship for matrices of
finite size?
%
%
By dropping limits, and using orthogonal polynomials to represent
relevant probability densities associated with finite random matrices,
we see that the semicircle associated with the \GUE{} emerges from an
average of two quarter-circles.

For a distribution on $n$-sets, the $m$-point correlation function,
$\sigma_n(x_1,x_2,\dotsc,x_m)$ describes the induced distribution on
uniformly selected subsets of size $m\leq{}n$.  By convention,
$\sigma_n(x_1,x_2,\dotsc,x_m)$ is not a probability distribution, but
is instead normalized such that
\[
\int_{\mathbb{R}^m}\sigma_n(x_1,x_2,\dotsc,x_m)\,\mathrm{d}x_1\dotsb\mathrm{d}x_m=\frac{n!}{(n-m)!}.
\]
Conceptually, when the underlying random process generates a single
unordered $n$-set, it can be thought of as producing
$m!\binom{n}{m}=\frac{n!}{(n-m)!}$ corresponding ordered $m$-tuples.
We will be interested primarily in $\frac{1}{n}\sigma_n(x)$, which
describes the pdf of a uniformly selected $1$-set.  When the
distribution on the $n$-sets takes the form
\[
p_n(x_1,x_2,\dotsc,x_n)=c_n\prod_{1\leq{}i<j\leq{}n}(x_i-x_j)^2\prod_{i=1}^n{}w(x_i)\,\mathrm{d}x_i,
\]
as with the \GUE{} ($w(x)=\mathrm{e}^{-x^2/2}$) and the \LUE{}
($w(x)=x^a\,\mathrm{e}^{-x/2}$), the $m$-point correlation function,
$\sigma_n(x_1,x_2,\dotsc,x_m)$, is given by an $m\times{}m$
determinant
\begin{equation}\label{eqn:m-pointcorrelation}
  \begin{split}
    \sigma_n(x_1,x_2,\dotsc,x_m)
    &=c_n\frac{n!}{(n-m)!}\int\prod_{1\leq{}i<j\leq{}n}(x_i-x_j)^2\prod_{i=1}^nw(x_i)\,
    \mathrm{d}x_{m+1}\mathrm{d}x_{m+2}\dotsb\mathrm{d}x_{n}\\ 
    &=\det\left(K_2(x_i,x_j)\right)_{1\leq{}i,j\leq{}m}
  \end{split}
\end{equation}
where
$K(x,y)=\sqrt{w(x)w(y)}\sum_{j=0}^{n-1}\varphi_j(x)\varphi_j(y)$, and
$\{\varphi_{j}(x)\colon j\geq0\}$ are orthonormal polynomials
associated with the weight $w(x)$ such that $\varphi_j(x)$ has degree
$j$ and $\int\varphi_i(x)\varphi_j(x)w(x)\,\mathrm{d}x=\delta_{i,j}$.
This result is based on the fact that the Vandermonde matrix can be
expanded in terms of any monic polynomials, and the resulting
integrals can be evaluated column by column, and can be conceptualized
as a generalization of the Cauchy-Binet formula to matrices of
continuous dimension.  A more complete discussion can be found, for
example, in \cite[Sec.~5.4]{Deift-OP} or \cite[Ch.~5]{Mehta}. 

For the \GUE{}, $w(x)=\mathrm{e}^{-x^2/2}$, and the functions
$\varphi$ are related to probabilists' Hermite polynomials described
by the initial conditions $H_0(x)=1$ and $H_1(x)=x$, and by the
$3$-term recurrence $H_{k+1}(x)=xH_k(x)-kH_{k-1}(x)$ for $k\geq1$.
%
%
It follows from the evaluation
\[
\int_{-\infty}^{\infty}H_i(x)H_j(x)\,\mathrm{e}^{-x^2/2}\;\mathrm{d}x=\delta_{i,j}\,n!\sqrt{2\pi},
\] 
that the level density, describing the probability density function
for the distribution of a single eigenvalue selected uniformly from
the eigenvalues of the order $n$ \GUE{} is given by
\begin{equation}\label{eqn:HermLD}
\frac{1}{n}\sigma_n^\mathrm{H}(x)
=\frac{1}{n\sqrt{2\pi}}\sum_{k=0}^{n-1}\frac{H_k(x)^2}{k!}\,\mathrm{e}^{-x^2/2}.
\end{equation}
Applying the Christoffel-Darboux formula to the sum provides the
compact representation
\[
\frac{1}{n}\sigma_n^\mathrm{H}(x)=\frac{1}{n!\sqrt{2\pi}}\big(H_n(x)^2-H_{n-1}(x)H_{n+1}(x)\big)\,\mathrm{e}^{-x^2/2},
\]
from which the eponymous semicircle law can then be recovered using
asymptotic properties of Hermite polynomials, as in
\cite[Appendix~A.9]{Mehta}.  Figure~\ref{fig:HermPlot} shows the level
density for the $7\times7$ \GUE{} as an approximation of a semicircle.
The relationship to \LUE{}s and the quarter-circle law will be
observed by considering the even and odd terms of the summand in
\eqref{eqn:HermLD} separately.

\begin{figure}
\hspace*{\fill}\includegraphics{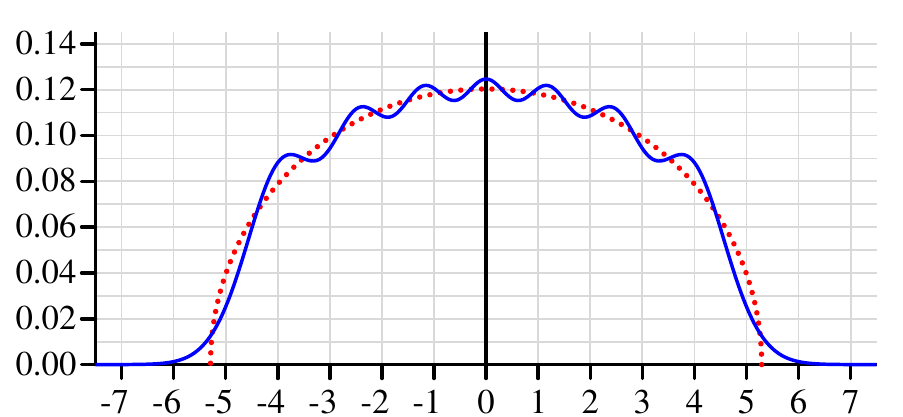}\hspace{\fill}\mbox{}
\caption{A semi-ellipse (\textcolor{BrickRed}{red}), with equation
  $y=\frac{1}{\sqrt{7}\,\pi}{\sqrt{1-\frac{x^2}{4\cdot7}}}$, is
  approximated by the probability density function for the
  distributions of a uniformly selected eigenvalue of the $7\times7$
  \GUE{} (\textcolor{blue}{blue}) as given by
  $\frac{1}{7}\sigma_7(x)
=\frac{1}{7\sqrt{2\pi}}\sum_{k=0}^{6}\frac{H_k(x)^2}{k!}\mathrm{e}^{-x^2/2}$.
}
\label{fig:HermPlot}
\end{figure}

For the \LUE{}, $w(x)=x^a\,\mathrm{e}^{-x/2}$, and the relevant
$\varphi$ can be expressed in terms of
$\{L_k^{(a)}(x)\colon{}k\geq0\}$, the generalized Laguerre polynomials
of parameter $a$, which satisfy
\begin{align*}
  \int_0^{\infty} L_i^{(a)}(x)L_j^{(a)}(x)\,
  x^{a}\mathrm{e}^{-x}\,\mathrm{d}x = \frac{\Gamma(j+a+1)}{j!} \delta_{i,j}.
\end{align*}
Although not required in the present context, it is convenient to note
that the Laguerre polynomials are given explicitly by
$L^{(a)}_n(x)=\sum_{i=0}^n(-1)^i\binom{n+a}{n-i}\frac{x^i}{i!}$.  The
weight function for the Laguerre polynomials differs by a factor of
$2$ in the exponential from our normalization of the \LUE{}, and this
is the source of the rescaled parameters in the subsequent formulae.
When $a$ is a half-integer, we can write $\Gamma(j+a+1)$ in terms of
factorials, and obtain
\begin{align*}
\sigma_{n}^{\mathrm{L}^-}(x)&=
          \frac{1}{\sqrt{2\pi}}
          \sum_{k=0}^{n-1}\frac{4^kk!^2}{(2k)!}
          \Big[L_k^{(-1/2)}\big(\frac{x}{2}\big)\Big]^2\frac{1}{\sqrt{x}}\,\mathrm{e}^{-x/2}\\
\sigma_{n}^{\mathrm{L}^+}(x)&=
          \frac{1}{\sqrt{2\pi}}
          \sum_{k=0}^{n-1}\frac{4^kk!^2}{(2k+1)!}
          \Big[L_k^{(+1/2)}\big(\frac{x}{2}\big)\Big]^2\sqrt{x}\,\mathrm{e}^{-x/2},
\end{align*}
corresponding to $a=-1/2$ and $a=+1/2$, with both functions supported
on the positive real axis.  The pdf for the distribution of a
uniformly selected singular values is thus given by
$\frac{2}{n}y\sigma_n^{\mathrm{L}^\pm}(y^2)$ (the extra factor of $2y$
is because the density is associated with an implicit differential, so
the change of variable $x=y^2$ also induces the substitution
$\mathrm{d}x=2y\,\mathrm{d}y$).  By the Marchenko-Pastur law, both
$\frac{2}{n}y\sigma_n^{\mathrm{L}^-}(y^2)$ and
$\frac{2}{n}y\sigma_n^{\mathrm{L}^+}(y^2)$ converge to
quarter-ellipses as $n\to\infty$.

To see the relationship between the semicircle law and the
quarter-circle law, we note that the left-right symmetry of
$\sigma_n^{\mathrm{H}}(x)$ is a consequence of the fact that the
matrix entries of the \GUE{} are distributed symmetrically about the
origin, so that the density of a matrix is identical to the density of
its negation.  The semicircle law is thus an example of a property
that nominally involves eigenvalues, but can instead be analyzed in
terms of singular values: it is sufficient to show a relationship
between $\sigma_{n}^{\mathrm{H}}$ and $\sigma_n^{\mathrm{L}^{\pm}}$
for positive arguments. 

Hermite polynomials, $H_k(x)$, are either even or odd polynomials,
according to the parity of $k$.  As a consequence, $H_{2m}(\sqrt{y})$
and $\frac{1}{\sqrt{y}}H_{2m+1}(\sqrt{y})$ are both monic polynomials
of degree $m$ in $y$.  Applying the change of variables $y=x^2$ to the
orthogonality relationship for Hermite polynomials, and using the fact
that $H_i(x)H_j(x)$ is an even polynomial whenever $i$ and $j$ have
the same parity, we see that on restricting to the positive $x$-axis
\[
\int_{0}^{\infty}H_{2m+r}\left(\sqrt{y}\right)H_{2l+r}\left(\sqrt{y}\right)\,y^{-\frac12}\mathrm{e}^{-\frac{y}{2}}\,\mathrm{d}y=\delta_{l,m}\,(2m+r)!\sqrt{2\pi}.
\]
So the monic polynomials $\{H_{2m}(\sqrt{y})\colon{}m\geq0\}$ are
orthogonal on $(0,\infty)$ relative to
$y^{-\frac12}\mathrm{e}^{-\frac{y}{2}}$, and similarly
$\{\frac{1}{\sqrt{y}}H_{2m+1}(\sqrt{y})\colon{}m\geq0\}$ are
orthogonal on $(0,\infty)$ relative to
$y^{\frac12}\mathrm{e}^{-\frac{y}{2}}$.  As a consequence we recover
the classical fact (see for example \cite[Sec.~5.6]{Szego-OP}) that Hermite
polynomials can also be expressed in terms of the generalized Laguerre
polynomials.  In particular, even ($r=0$) and odd ($r=1$) Hermite
polynomials are given by the expression
\begin{equation}\label{eqn:HermiteasLaguerre}
H_{2n+r}(x)=x^r(-2)^n\,n!\,L_n^{(r-1/2)}\big(\frac{x^2}{2}\big) \quad\text{for $r\in\{0,1\}$},
\end{equation}
Substituting \eqref{eqn:HermiteasLaguerre} into \eqref{eqn:HermLD} we
find
\begin{align*}
  \sigma_n^{\mathrm{H}}(x)
  &=\frac{1}{\sqrt{2\pi}}\sum_{k=0}^{n_1}\frac{4^kk!^2}{(2k)!}
                        \Big[L_{k}^{(-1/2)}\big(\frac{x^2}{2}\big)\Big]^2\mathrm{e}^{-x^2/2}
  +\frac{1}{\sqrt{2\pi}}\sum_{k=0}^{n_2}\frac{4^kk!^2}{(2k+1)!}
                        \Big[L_{k}^{(+1/2)}\big(\frac{x^2}{2}\big)\Big]^2x^2\,\mathrm{e}^{-x^2/2}\\
  &=x\sigma_{n_1}^{\mathrm{L}^-}(x^2)+x\sigma_{n_2}^{\mathrm{L}^+}(x^2)
\end{align*}
where $n_1=\ceil{\frac{n}{2}}$ and $n_2=\floor{\frac{n}{2}}$.  An
immediate consequence is that the restriction of the semicircle law to
the positive quadrant is manifestly a weighted average of two copies
of the quarter-circle law.  Since
$\frac{2}{n_1}x\sigma_{n_1}^{L^-}(x^2)$ and
$\frac{2}{n_2}x\sigma_{n_2}^{L^+}(x^2)$ describe the pdfs of
uniformly chosen singular values of bi-diagonal Laguerre matrices, we
conclude that a single singular value of a \GUE{} matrix has the same
distribution as a single singular value selected from the direct sum
of two \LUE{} matrices, or via the equivalence of \LUE{} and
anti-\GUE{} matrices two anti-\GUE{} matrices of consecutive sizes.

\begin{figure}
\includegraphics[width=\textwidth]{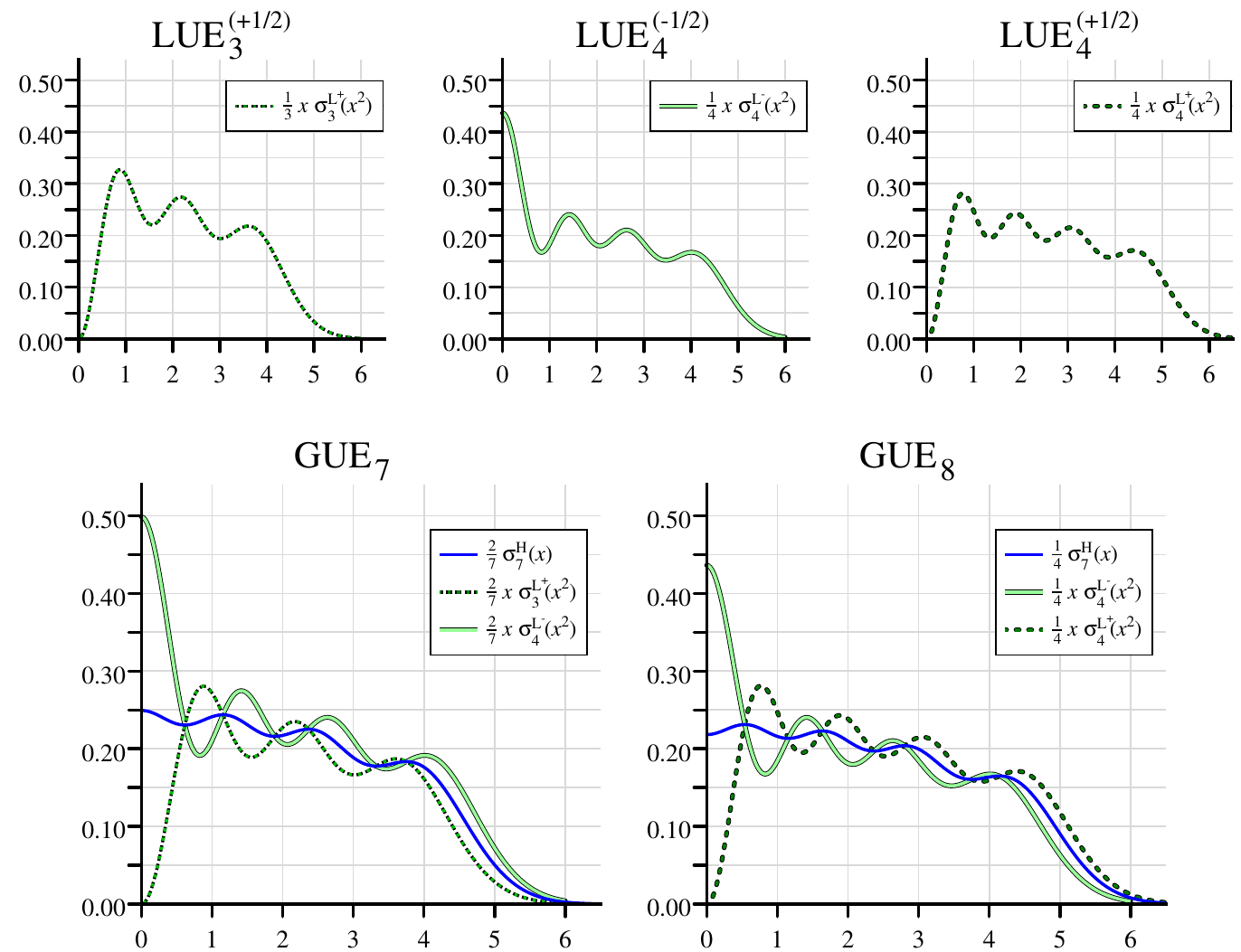}
\caption{The top three plots show pdfs for distributions of the
  positive square root of an eigenvalue selected uniformly at random
  from various Laguerre ensembles.  The bottom two plots give the
  distributions for uniformly selected singular values of \GUE{}s
  (\textcolor{blue}{blue}) as a weighted average of the Laguerre
  pdfs (\textcolor{OliveGreen}{green}).  For odd $n$, the weighting is
  unequal.  For example,
  $\frac27\sigma_7^\mathrm{H}(x)=\frac{1}{2}\left(
  \frac{4}{7}x\sigma_3^{\mathrm{L}^+}(x^2)+\frac{4}{7}x\sigma_4^{\mathrm{L}^-}(x^2)
  \right)$, but it is $\frac{2}{3}x\sigma_3^{\mathrm{L}^+}(x^2)$ and
  $\frac{2}{4}x\sigma_4^{\mathrm{L}^-}(x^2)$ that are probability
  densities.  When $n$ is even the weighting is equal, for example,
  $\frac28\sigma_8^\mathrm{H}(x)=\frac{1}{2}\left(
  \frac{2}{4}x\sigma_4^{\mathrm{L}^+}(x^2)+\frac{2}{4}x\sigma_4^{\mathrm{L}^-}(x^2)
  \right)$.}\label{fig:LagPlots}
\end{figure}

\begin{rem}
We see that the $1$-point correlation function for the singular
values of the \GUE{} is a sum of two $1$-point correlation functions
for Laguerre singular values, and conclude that picking a random
singular value from the \GUE{} is equivalent to picking one from a
mixture of two appropriate \LUE{}s.  It is notable that each
\LUE{} distribution is a component of two consecutive \GUE{}
distributions.  Figure~\ref{fig:LagPlots} presents the level densities
for the $7\times7$ and $8\times8$ \GUE{}'s as weighted averages of
\LUE{} singular value level densities.  Notice that the
$\LUE_4^{(+1/2)}$ distribution is a component of both mixtures.
\end{rem}

Concretely, we conclude that a uniformly random singular value of the
$n\times{}n$ \GUE{} has the same distribution as uniformly random
singular value from the union of an $n\times{}n$ anti-\GUE{} and an
$(n+1)\times(n+1)$ anti-\GUE{}.  In fact, it is possible to continue
in this direction, as Forrester did in \cite{Forrester-Evenness}, to
effectively show the analogous property for all $m$-point correlation
functions, and conclude that the singular values of the \GUE{} are a
mixture of the singular values of two independent ensembles.  In the
next section, we take a different approach, and derive the same
conclusion without appealing to the theory of orthogonal polynomials.

\goodbreak

\section{Main Result - A Decomposition of the Singular Values of the \GUE{}}\label{sec:decomposition}

The main result of the paper equates the joint probability density
functions for two distributions, the singular values of the
$n\times{}n$ \GUE{} and the union of the distinct non-zero singular
values of two independent anti-\GUE{} matrices, one of order $n$ and
the other of order $(n+1)$.  We do not know of any particularly
compact descriptions for the joint pdf for the distribution of the
singular values in either setting, but instead express each as a sum.
In terms of the $n\times{}n$ \GUE{}, this sum involves $2^n$ terms,
corresponding to the number of ways that $n$ singular values can be
assigned signs.  In contrast, there are $\binom{n}{\floor{{n}/{2}}}$
ways the singular values can be partitioned between an $n\times{}n$
anti-\GUE{} and a $(n+1)\times{}(n+1)$ anti-\GUE{}.  Neither sum is
particularly compact, but the second involves asymptotically fewer
terms by a factor of $\sqrt{2\pi{}n}$.  The result is thus to be
interpreted primarily as structural: it is this extra structural
information, rather than the expressions themselves, that can be used
to computational advantage.

Our main tool is to express relevant probability densities in terms of
determinants, and then to recognize evaluations that induce a
structured sparsity and permit writing the resulting determinants as
products.  In addition to appearing in Forrester's work on gap
probabilities, where \cite[Eq.~(2.6)]{Forrester-Evenness} is
equivalent to our main result, the same pattern occurs in existing
proofs of the applications discussed in Section~\ref{sec:applications}
(\cite[Ch.~20]{Mehta} and \cite{MeNo}), as well as related problems
about enumerative properties of orientable maps (\cite{JV-Characters,
  JV-Eulerian, JPV}).  The emphasis of the present work is to show
that all of these results are consequences of the same underlying
structural decomposition.

In practice, when evaluating integrals of symmetric functions, it is
often convenient to write the integral as a sum of terms that are
equal by symmetry, and then to consider only a single term.  This is
the case, for example, in \cite[Ch.~15]{Mehta}), where Mehta considers
integrals related to complex Ginibre ensembles.  We give presentations
of \eqref{eqn:GUEdensity} and \eqref{eqn:skewdensity} that are
suitable for desymmetrization.

\begin{lemma}\label{lem:nonsymGUE} 
  The joint probability density for the eigenvalues of the
  \GUE{} can be represented as
  \begin{equation}\label{eqn:nonsymGUE}
    p_n^{\mathrm{H}}(x_1,x_2,\dotsc,x_n) =
       c_n^{\mathrm{H}}\sum_{\pi\in\mathfrak{S}_n}
       {\det\left(\big(x_{\pi_i}^{i+j-2}\big)_{1\leq i,j\leq n}\right)}
       \,\exp\Big(-\frac12\sum_{i=1}^n x_i^2\Big)
       \differentials{x}{1}{n},
  \end{equation}
  where the sum is taken over all permutations of the indices $\{1,2,\dotsc,n\}$.
\end{lemma}

\begin{proof}
  We recognize the product $\prod_{1\leq{}i<j\leq{}n}\abs{x_i-x_j}^2$
  in \eqref{eqn:GUEdensity} as the square of the Vandermonde
  determinant, and use multiplicativity and invariance under
  transposition to obtain
  \begin{align*}
    {\prod_{1\leq i<j\leq n}\abs{x_i-x_j}^2}
    &={\det\left(\big(x_i^{j-1}\big)_{1\leq i,j\leq
          n}\big(x_j^{i-1}\big)_{1\leq i,j\leq n}\right)}.
  \end{align*}
  Multiplying the matrices inside the determinant gives the expression.
  \begin{align*}
    \prod_{1\leq{}i<j\leq{}n}\abs{x_i-x_j}^2&=
    {\det\begin{pmatrix} n & p_1(\vec{x}) & p_2(\vec{x})
      & \cdots & p_{n-1}(\vec{x}) \\ p_1(\vec{x}) & p_2(\vec{x}) &
      p_3(\vec{x}) & \cdots & p_{n}(\vec{x}) \\ p_2(\vec{x}) &
      p_3(\vec{x}) & p_4(\vec{x}) & \cdots & p_{n+1}(\vec{x})
      \\ \vdots & \vdots & \vdots & \ddots & \vdots
      \\ p_{n-1}(\vec{x}) & p_n(\vec{x}) & p_{n+1}(\vec{x}) & \cdots &
      p_{2n-2}(\vec{x}).
      \end{pmatrix}}
  \end{align*}
  By column-linearity, we can write this as a sum of determinants
  where each column depends on a single variable.  The determinant
  vanishes if any variable is duplicated, so the sum can be restricted
  to terms in which each variable occurs in a single column, and we
  obtain
  \begin{align*}
    \prod_{1\leq{}i<j\leq{}n}\abs{x_i-x_j}^2&=
    \sum_{\pi\in\mathfrak{S}_n}
       {\det\left(\big(x_{\pi_i}^{i+j-2}\big)_{1\leq i,j\leq n}\right)}.
  \end{align*}
  The factor, $\exp\left(-\frac12\sum_{i=1}^nx_i^2\right)$ is
  symmetric and thus common to all terms, and the result follows.
\end{proof}

The proof relied only on the presence of the Vandermonde factor, so a
similar derivation gives a corresponding presentation for the joint
probability density functions of the anti-\GUE{} given in
\eqref{eqn:skewdensity}.

\begin{lemma} For $r\in\{0,1\}$, the joint pdf of the
$(2n+r)\times(2n+r)$ anti-\GUE{} can be presented as
  \begin{equation}\label{eqn:nonsymaGUE}
    p_{2n+r}^{\mathrm{aG}}(x_1,x_2,\dotsc,x_n)\coloneqq{}
    c_{2n+r}^{\mathrm{aG}}\sum_{\pi\in\mathfrak{S}_n}
    {\det\left(\big(\theta_{\pi_i}^{2i+2j-4+2r}\big)_{1\leq i,j\leq n}\right)}
    \,\exp\big(-\frac{1}{2}\sum_{j=1}^n\theta_j^2\big)
    \prod_{j=1}^n\mathrm{d}\theta_i.
  \end{equation}
\end{lemma}

The identity terms of the summations in \eqref{eqn:nonsymGUE} and
\eqref{eqn:nonsymaGUE} determine signed measures,
$\mu_{n}^{\mathrm{H}}$ and $\mu_{n}^{\mathrm{aG}}$, on ordered
$n$-tuples of real numbers given by
\begin{align*}
  \mu_{n}^{\mathrm{H}}(x_1,x_2,\dotsc,x_n)&\coloneqq n!\,c_n^{\mathrm{H}}
  {\det\left(\big(x_{i}^{i+j-2}\big)_{1\leq i,j\leq n}\right)}
  \,\exp\Big(-\frac12\sum_{i=1}^n x_i^2\Big)
  \differentials{x}{1}{n},\\
  \mu_{2n+r}^{\mathrm{aG}}(\theta_1,\theta_2,\dotsc,\theta_n)&\coloneqq
    n!\,c_{2n+r}^{\mathrm{aG}}
    {\det\left(\big(\theta_{i}^{2i+2j-4+2r}\big)_{1\leq i,j\leq n}\right)}
    \,\exp\big(-\frac{1}{2}\sum_{j=1}^n\theta_j^2\big)
    \prod_{j=1}^n\mathrm{d}\theta_i,
\end{align*}
that can be used to evaluate expectations of symmetric functions of
eigenvalues.  Unlike $p_n^{\mathrm{H}}$ and $p_{2n+r}^{\mathrm{aG}}$
which can be sampled by diagonalizing randomly generated matrices (for
example the tri-diagonal matrices of \cite{DuEd} or \cite{Trotter}),
it is not clear how to use $\mu_{n}^{\mathrm{H}}$ and
$\mu_{2n+r}^{\mathrm{aG}}$ in Monte Carlo experiments.

\goodbreak

We are now in a position to give the main result of the paper, an
equivalent form of which was previously described by Forrester in
\cite[Eq.~(2.6)]{Forrester-Evenness}.  

\begin{thm}\label{thm:mainresult}
  The singular values of the $n\times{}n$ \GUE{} have the same
  distribution as the union of the distinct singular values of two
  anti-\GUE{} matrices, one of order $n$, the other of order $n+1$.
  In terms of probability densities, for $x_1,x_2,\dotsc,x_n\geq0$,
  \begin{equation}\label{eqn:maintoprove}
    \sum_{\epsilon\in\{\pm1\}^n}p_n^{\mathrm{H}}(\epsilon_1x_1,\epsilon_2x_2,\dotsc,\epsilon_nx_n)
    =
    \frac{1}{\binom{n}{\floor{n/2}}}
    \sum_{S,T}p_{n}^{\mathrm{aG}}(x_{s_1},x_{s_2},\dotsc,x_{s_{\floor{n/2}}})
    p_{n+1}^{\mathrm{aG}}(x_{t_1},x_{t_2},\dotsc,x_{t_{\ceil{n/2}}})
  \end{equation}
  where the left sum runs over the $2^n$ ways $\{x_1,x_2,\dotsc,x_n\}$
  can be assigned signs to describe eigenvalues of a \GUE{}, and the right
  sum runs over the $\binom{n}{\floor{n/2}}$ ways to partition
  $\{1,2,\dotsc,n\}$ into two sets
  $S=\{s_1,s_2,\dotsc,s_{\floor{n/2}}\}$ and
  $T=\{t_1,t_2,\dotsc,t_{\ceil{n/2}}\}$ corresponding to the
  anti-\GUE{} factors.
\end{thm}

\begin{proof}
  Using Lemma~\ref{lem:nonsymGUE}, the density of the singular values
  of the \GUE{} is given by
  \begin{equation*}
    c_n^{\mathrm{H}}
    \sum_{\epsilon\in\{\pm1\}^n}
    \sum_{\pi\in\mathfrak{S}_n}
    {\det\left(\big((\epsilon_{i}x_{\pi_i})^{i+j-2}\big)_{1\leq
        i,j\leq n}\right)} \,\exp\Big(-\frac12\sum_{i=1}^n x_i^2\Big)
    \differentials{x}{1}{n}.
  \end{equation*}
  Both sums are finite, so the $\epsilon$-summation can be carried out
  first, and since for fixed $\pi$ each $\epsilon_i$ occurs in a
  single column, it can be applied column-by-column.  Except for the
  determinant, the expression is invariant under sign-change, so
  summing over $\epsilon$ annihilates all matrix entries involving
  monomials of odd degree and creates a checkerboard pattern of
  sparsity.  The rows and columns of the resulting matrix can be
  simultaneously permuted so that the odd columns and rows occur
  before the even columns and rows, and this collects the zero and
  non-zero entries into blocks (see Figure~\ref{fig:shuffles}).  
  \begin{figure}
    \[
    \left(\gminipage[scale=.9]{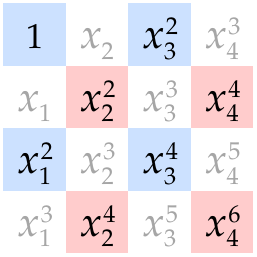}\right)\sim\left(\gminipage[scale=.9]{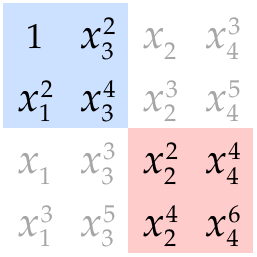}\right)
    \quad
    \left(\gminipage[scale=.9]{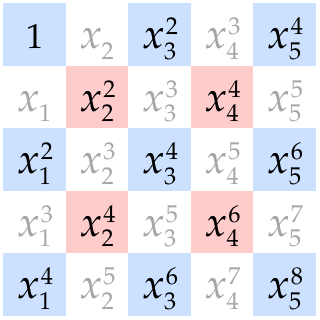}\right)\sim\left(\gminipage[scale=.9]{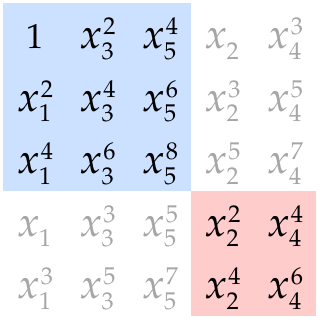}\right)
    \]
    \caption{The rows and columns of
      $\big(x_i^{i+j-2}\big)_{1\leq{}i,j\leq{}n}$ are permuted to
      group the even monomials into two blocks. 
      One block involves even-indexed variables, the
      other involves odd-indexed variables.  Odd monomials are
      annihilated by the $\epsilon$-summation and do not contribute to
      the density of singular values.  }\label{fig:shuffles}
  \end{figure}
  We need only consider the direct sum of the non-zero blocks, and 
  obtain the decomposition:
  \[
  \sum_{\epsilon\in\{\pm1\}^n}
  \det\left(\big((\epsilon_ix_{\pi_i})^{i+j-2}\big)_{1\leq{}i,j\leq{}n}\right)
  =2^n\det\left( \big(x_{\pi_{2i-1}}^{2i+2j-4}\big)_{1\leq i,j\leq
    \ceil{\frac{n}{2}}} \oplus \big(x_{\pi_{2i}}^{2i+2j-2}\big)_{1\leq
    i,j\leq \floor{\frac{n}{2}}} \right).
  \]
  The determinant on the right-hand side factors as a product of two
  determinants of the type occurring in \eqref{eqn:nonsymaGUE} for
  anti-\GUE{} of orders $2\ceil{\frac{n}{2}}$ and
  $2\floor{\frac{n}{2}}+1$, which correspond to $n$ and $n+1$ for all
  choices of $n$ (although the particular pairing depends on parity).
  The exponential factor and volume elements are separable, and thus
  the theorem is true up to a scalar factor.  This scalar is
  necessarily unity since both sides of \eqref{eqn:maintoprove} are
  probability densities on the positive orthant.
\end{proof}

  \goodbreak

\begin{exam}\label{exam:mu2}
  For the \GUE{} of order $2$, we have
  $p_2^{\mathrm{H}}(x,y)=\frac{1}{4\pi}(x^2-2xy+y^2)\mathrm{e}^{-x^2/2-y^2/2}$
  and
  $\mu_2^{\mathrm{H}}(x,y)=\frac{1}{2\pi}(y^2-xy)\mathrm{e}^{-x^2/2-y^2/2}$.
  Since
  $p_2^{\mathrm{H}}(a,b)=\frac{1}{2}\mu_2^{\mathrm{H}}(a,b)+\frac{1}{2}\mu_2^{\mathrm{H}}(b,a)$
  for all $a$ and $b$, they induce the same density on $2$--sets.
  While $p_2^{\mathrm{H}}(x,y)$ is symmetric about $y=x$ and $y=-x$,
  the signed measure $\mu_2^{\mathrm{H}}(x,y)$ is symmetric about 
  $y=\frac12x$ and $y=-2x$ (see Figure~\ref{fig:GUE2pdfs}).
\end{exam}

\begin{figure}[!htbp]
\hspace*{\fill}%
\gminipage[width=.45\textwidth]{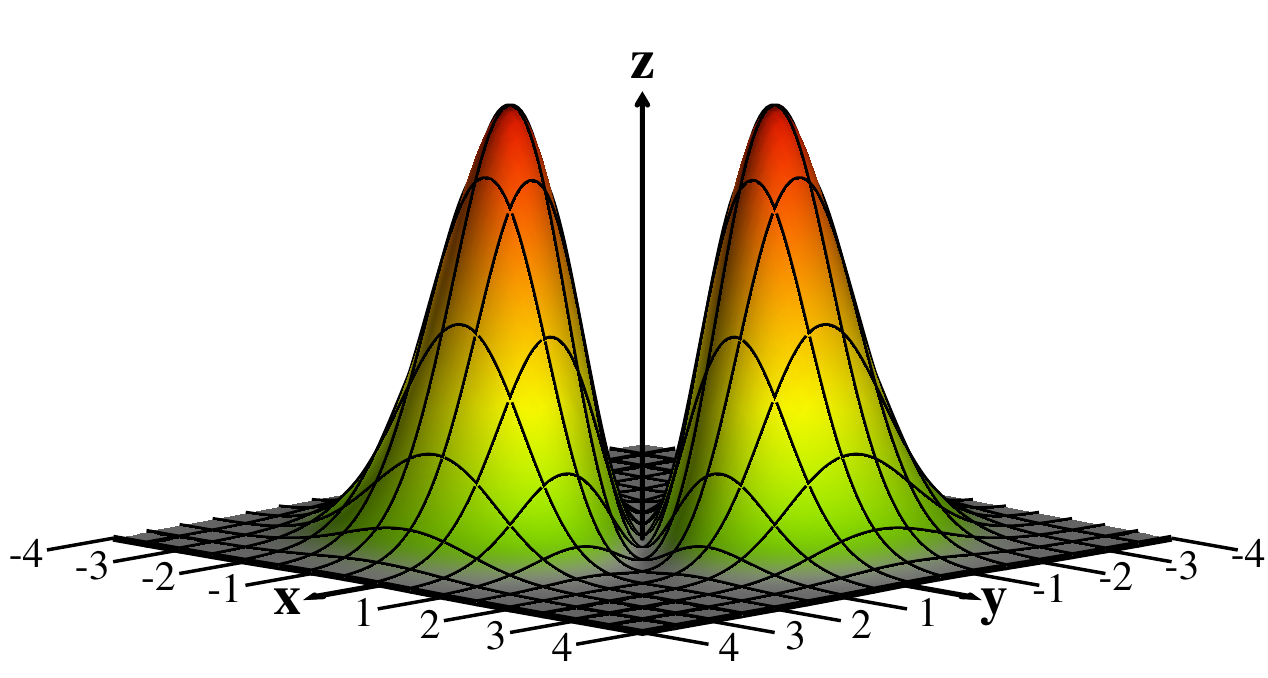}%
\hspace*{\fill}\mbox{}\hspace{\fill}%
\gminipage[width=.45\textwidth]{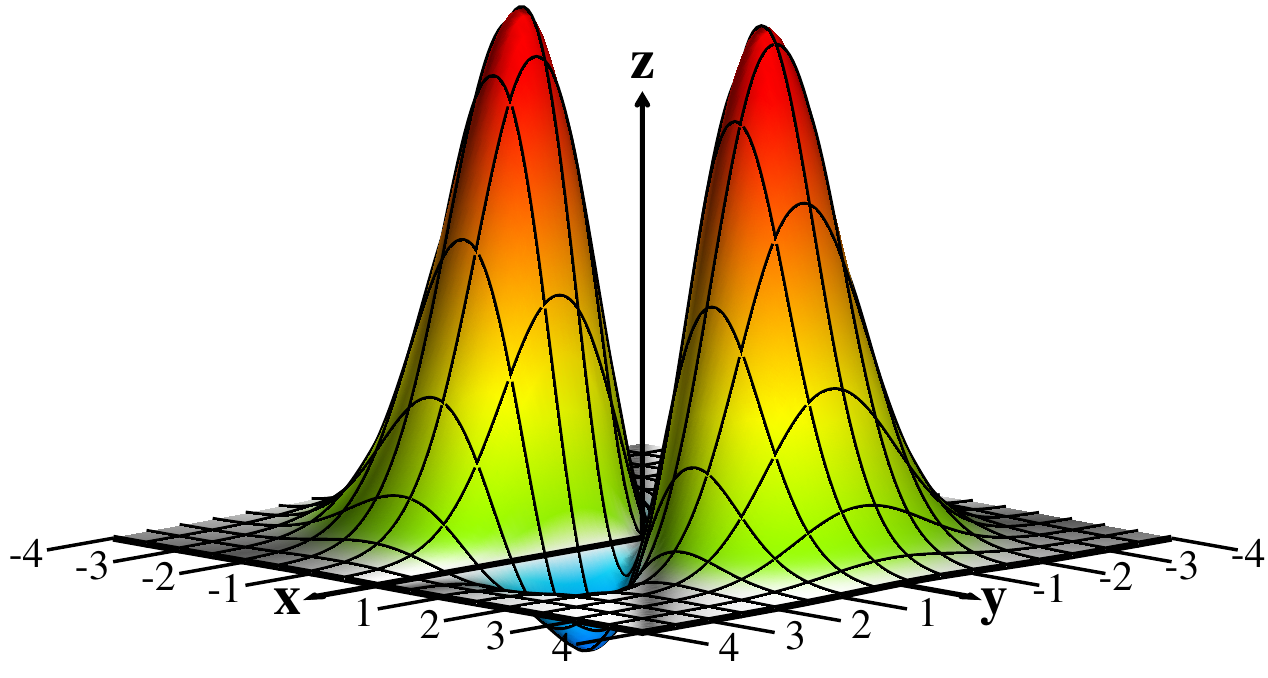}%
\hspace{\fill}\mbox{}
\hspace*{\fill}%
\gminipage[width=.36\textwidth]{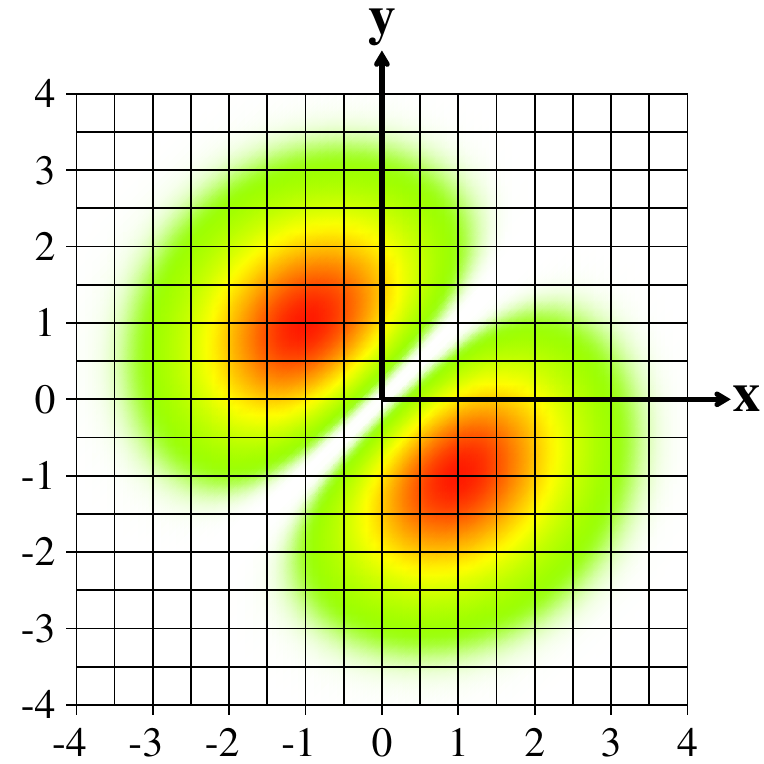}%
\hspace*{\fill}\mbox{}\hspace{\fill}%
\gminipage[width=.36\textwidth]{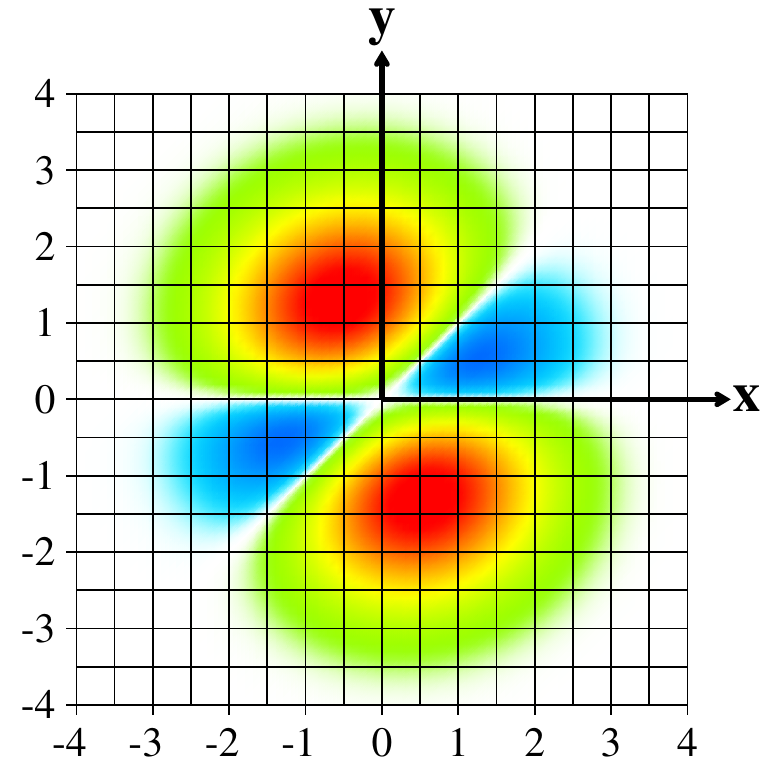}%
\hspace{\fill}\mbox{}
\caption{ The density $p_2^{\mathrm{H}}(x,y)$ (left) and signed
  measure $\mu_2^{\mathrm{H}}(x,y)$ (right).
}\label{fig:GUE2pdfs}
\end{figure}

\begin{exam}\label{exam:mu-sing-is-separable}
  The singular values of the $2\times2$ \GUE{} have density
  $\frac{1}{2}f(x,y)+\frac{1}{2}f(y,x)$ where
  \[
  f(x,y)\coloneqq\mu_2^{\mathrm{H}}(x,y)+\mu_2^{\mathrm{H}}(-x,y)+\mu_2^{\mathrm{H}}(x,-y)+\mu_2^{\mathrm{H}}(-x,-y)
  =\frac{2}{\pi}y^2\,\mathrm{e}^{-x^2/2-y^2/2}
  \]
  is itself a probability density.  It is the product of densities of
  independent $X\sim\chi_1$ and $Y\sim\chi_3$ random variables.
  Figure~\ref{fig:GUE2sing} plots
  $g(x,y)=\frac{f(x,y)}{f(0,\sqrt{2})}$ as the product
  $g(x,y)=g(x,\sqrt{2})\times{}g(0,y)$.
\end{exam}

\begin{figure}
\hspace*{\fill}\includegraphics[width=.7\textwidth]{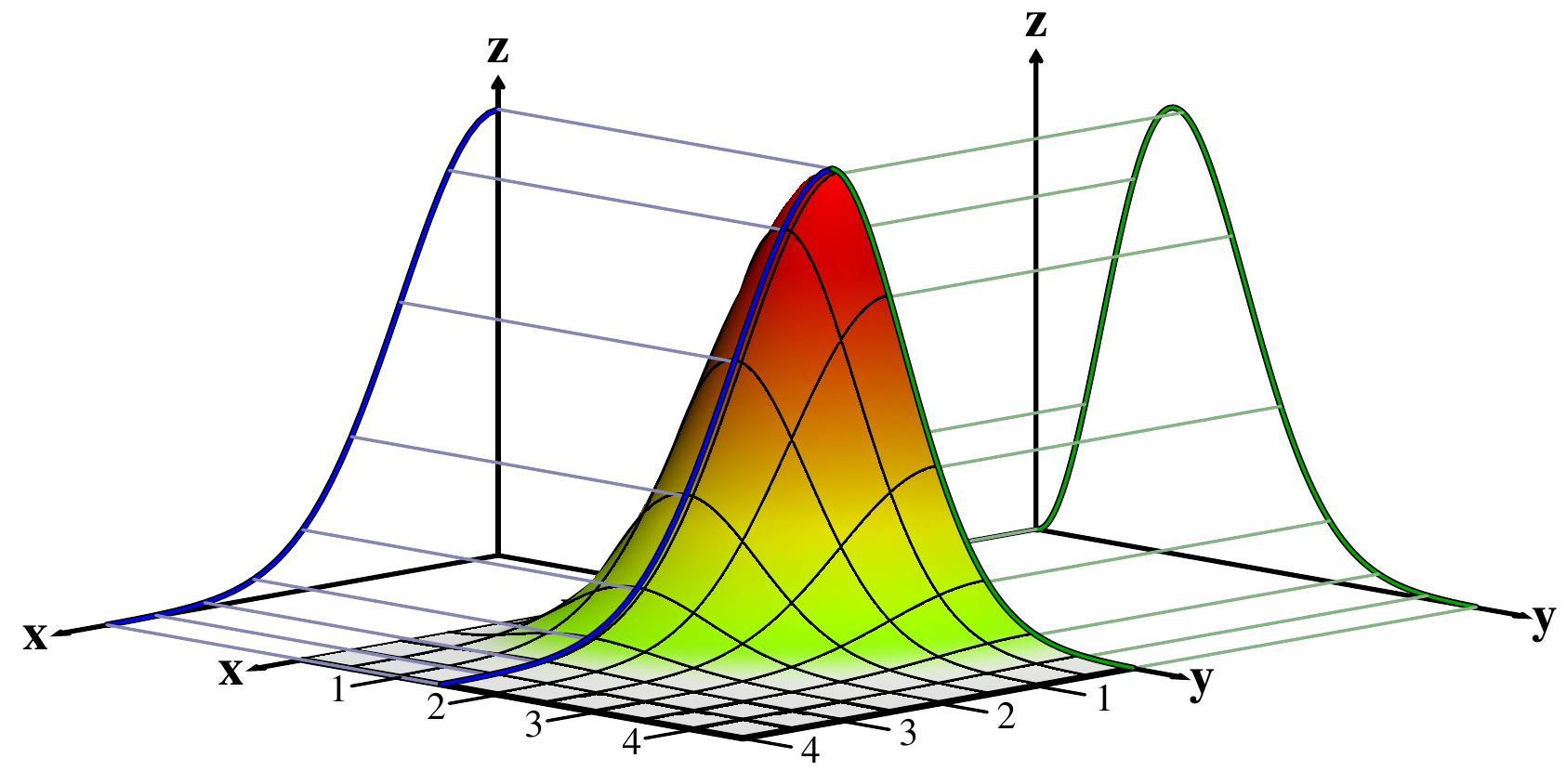}\hspace{\fill}\mbox{}
\caption{The function
  $g(x,y)=\frac{\mathrm{e}}{2}y^2\mathrm{e}^{-x^2/2-y^2/2}$ is
  separable with a maximum value of $g(0,\sqrt{2})=1$ and satisfies
  the relation $g(x,y)=g(x,\sqrt{2})\times{}g(0,y)$.  The factors
  correspond to the fact that the singular values of the $2\times{}2$
  \GUE{} have the same distribution as independently distributed
  $\chi_1$ and $\chi_3$ random variables.}\label{fig:GUE2sing}
\end{figure}

\begin{rem}
  The proof did not use the independent Gaussian nature of the matrix
  entries in any essential way.  As in
  \cite[Eq.~2.6]{Forrester-Evenness} and
  \cite[Sec~8.4.1]{Forrester-Log-gases}, a similar decomposition
  applies when working with any density of the form $c\prod_{1\leq
    i<j\leq n}(x_i-x_j)^2 \prod_{i=1}^nw(\abs{x_i})$.  The authors are
  unaware of any immediate applications involving these more general
  distributions.
\end{rem}

\begin{rem}
  In related work, previous authors have often been forced to invoke a
  case analysis, depending on the parity of $n$.  This is the case,
  for example, in the motivating work of Jackson and Visentin
  \cite{JV-Characters,JV-Eulerian}, in which their factorization
  applies directly only when $n$ is even, with odd $n$ recovered by
  polynomiality.  Avoiding this complication is the principal reason
  that we prefer to describe the component distributions in terms of
  anti-\GUE{} instead of $\LUE^{(\pm1/2)}$ distributions.
\end{rem}

\subsection{Bi-diagonal Representations}
A consequence of Theorem~\ref{thm:mainresult} is that the singular
values of the \GUE{} inherit bi-diagonal representations from anti-\GUE{}
and \LUE{} matrices.  In fact, for a given $n$, we have two distinct
bi-diagonal models for the singular values of the order $n$ \GUE{},
one using the natural representation of the odd anti-\GUE{} component
that emphasizes the one-parameter nature of the \GUE{}s, and a second
obtained by considering this component in terms of Laguerre--$(+1/2)$
matrices.  Figure~\ref{fig:bi-diagonalstaircase} shows the relationship
between these representations.  Using the \LUE{} representation of the
factors, we have the direct sum decomposition
\begin{equation}\label{eqn:GUEbidiag}
\GUE_n^{(s)}\sim\begin{pmatrix}
\chi_{n_1}\\\chi_{n_1-1}&\chi_{n_1-2}\\&\ddots&\ddots\\&&\chi_4&\chi_3\\&&&\chi_2&\chi_1
\end{pmatrix}
\oplus\begin{pmatrix}
\chi_{n_2}\\\chi_{n_2-3}&\chi_{n_2-2}\\&\ddots&\ddots\\&&\chi_4&\chi_5\\&&&\chi_2&\chi_3
\end{pmatrix}
\end{equation}
where $n_1=2\ceil{\frac{n}{2}}-1$ and $n_2=2\floor{\frac{n}{2}}+1$.
This gives even powers of the determinant of \GUE{} matrices as simple
a structure as determinants of Wishart matrices, and is the basis of
Section~\ref{sec:determinant}.

\begin{figure}
\hspace*{\fill}\includegraphics[scale=.9]{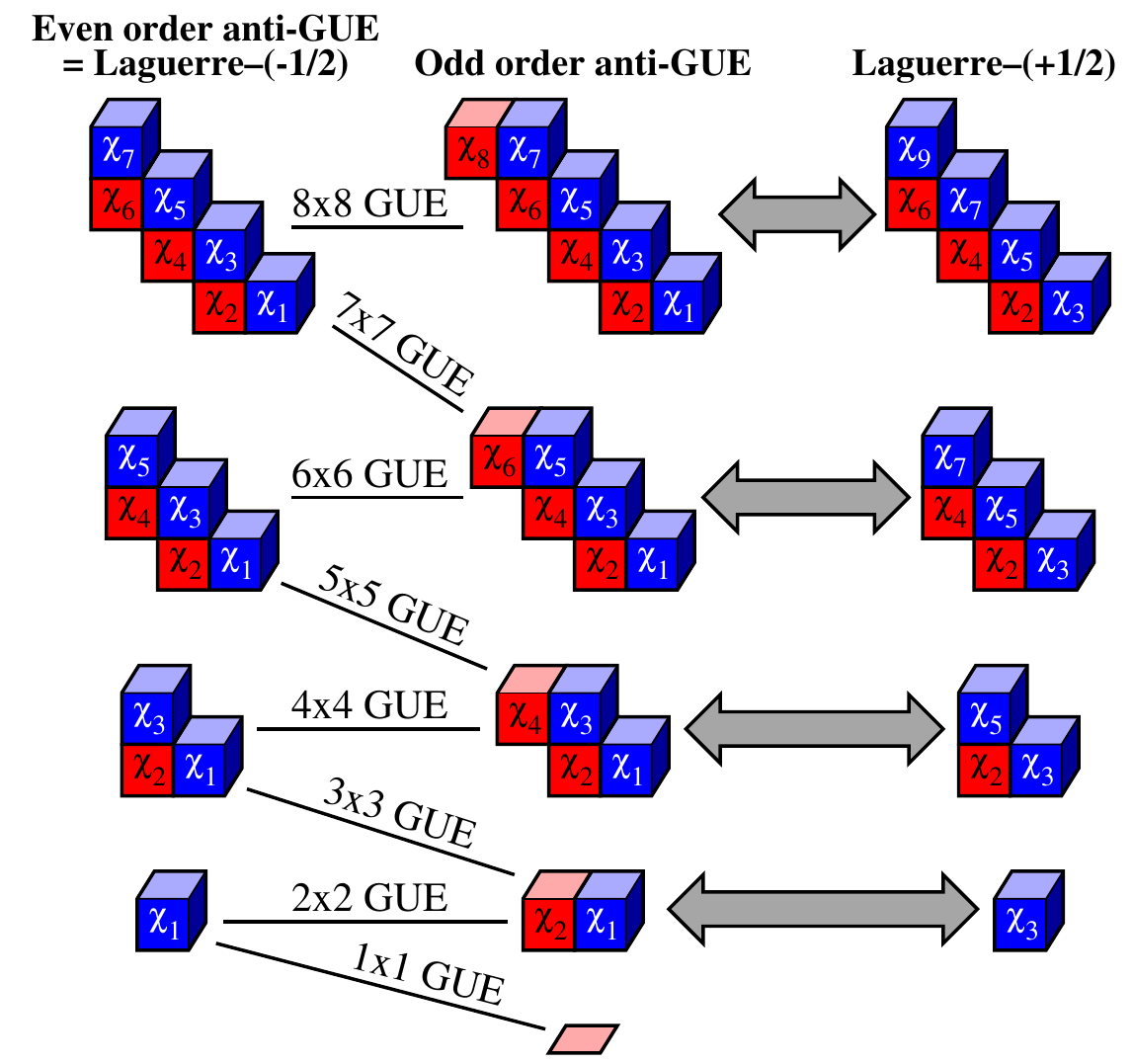}\hspace{\fill}\mbox{}
\caption{\label{fig:bi-diagonalstaircase} Bi-diagonal models for the
  singular values of \GUE{}s.  Each \GUE{} takes its singular values
  from two independent blocks corresponding to an even and an odd
  order anti-\GUE{}.  Each odd-order anti-\GUE{}s (second column) has
  an alternate representation when considered as a Laguerre--$(+1/2)$
  matrix (third column).}
\end{figure}

\subsection{Relationship with Combinatorics of Orientable Maps}\label{sec:maps}

Our interest in Theorem~\ref{thm:mainresult} stem from the second
author's attempts to develop a more natural understanding of what
turned out to be an equivalent result from enumerative combinatorics.
In fact both Theorem~\ref{thm:mainresult} and the relationship between
Hermite and Laguerre level densities from
Section~\ref{sec:leveldensities} can be stated as results about map
enumeration, although the equivalence is non-trivial.

Combinatorially, the moments of the $\GUE{}_n$ count rooted embeddings
of graphs in orientable surfaces in which each face is painted with
one of $n$ colors.  Similarly, when $n+a$ is a positive integer, the
moments of the $\LUE{}_n^{(a)}$ count rooted embeddings of
face-bipartite graphs in orientable surfaces such that each face of
one class is painted with one of $n$ colors, and each face of the
other class is painted with one of $m=n+a$ colors.  Scalings of these
moments for large $n$ are dominated by embeddings with a maximum
number of faces.  For planar maps, these are necessarily the duals of
trees, which are enumerated by Catalan numbers.  In this context, the
relationship between the semicircle law and the quarter-circle law can
be seen as a statement that all trees are bipartite, so it follows
that one-part moments of the two ensembles have the same leading order
behavior.




In \cite{JV-Characters}, Jackson and Visentin used permutation
representations of orientable maps to derive expressions for
generating series of several classes of such maps in terms of
irreducible characters of the symmetric group, and manipulated these
expressions \emph{via} their relationships with Schur functions.  In
this framework, they exhibited a sparsity pattern and factorization
for determinantal representations of products of the form
$V(\vec{x})^2s_{\theta}(\vec{x})$, with $s_\theta$ a Schur function,
and used this factorization to obtain a functional relationship
between generating series.  Their construction applies directly only
when $n$ is an even integer, but the functional identity is extended
to odd $n$ with the observation that moments depend polynomially on
$n$.  In contrast, by working with $V(\vec{x})^2$ directly in the
proof of Theorem~\ref{thm:mainresult}, we have been able to treat even
and odd values of $n$ simultaneously, and have avoided technical
manipulations of irreducible characters of the symmetric group.

The enumerative result was later interpreted in terms of matrix models
arising in the study of $2$--dimensional quantum gravity in
\cite{JPV}, and extended in \cite{JV-Eulerian} to a form equivalent to
Theorem~\ref{thm:mainresult}, although the authors were unaware of the
random matrix interpretation of this extension.  The generating series
they considered are effectively the cumulant generating functions for
the $\GUE{}_n$ and $\LUE{}_n^{(a)}$ densities, \eqref{eqn:GUEdensity}
and \eqref{eqn:laguerredensity}, taken to have functional dependence
on $n$ and $a$, although they did not interpret them in this way.  In
fact, the combinatorial interpretation for the cumulant generating
function for $\LUE{}_n^{(a)}$ can only be established directly when
$a$ is a non-negative integer, since $m=n+a$ is to be interpreted as
the cardinality of a set of colors.  The functional identity, however,
applies only when $a=\pm\frac12$.  This disconnect is resolved by
applying the observation that both the generating series for maps, and
the moments of the bi-diagonal model of the Laguerre ensemble,
\eqref{eqn:laguerredensity}, depend polynomially on $a$.

\goodbreak

\subsection{Observing the Decomposition}

Theorem~\ref{thm:mainresult} shows that it is impossible to
distinguish between the singular values of an $n\times{n}$ \GUE{} and
a mixture of the singular values of two anti-\GUE{} of appropriate
sizes.  A natural question is how this can be observed numerically,
and whether this kind of decomposition has a signature that can be
used to identify (or discount) related decompositions.  As we saw in
Examples~\ref{exam:mu2} and~\ref{exam:mu-sing-is-separable}, the
product structure is not visible directly for the pdf of the
$\GUE_2$, but only on a desymmetrized transform.

\definecolor{listingbackground}{gray}{.9}
\begin{lstlisting}[
    float, 
    language=Octave, 
    morekeywords={nchoosek}, 
    caption={MATLAB code for partitioning the singular values of $\GUE_n$},%
    label=lst:UnmixingSingularValues,%
    mathescape=true,
    captionpos=t,     % The caption is placed above the listing.
    frame=tlbr,       % Frame the listing on the [t]op, [l]eft, [b]ottom, and [r]ight
    frameround=,  % using rounded corners at all 4 corners
    columns=flexible, % fixed-width looks nicer, but the resulting pdf
                      % has spaces inserted between characters of
                      % words, and is unsuitable as a source for
                      % copy/paste.
    backgroundcolor=\color{listingbackground},
    commentstyle=\color{OliveGreen},
    keywordstyle=\color{Blue},
    % numberstyle=\tiny\color{Gray},  % Using line numbers creats a
    % numbers=left,                   % pdf from which source code
    % xleftmargin=.05\textwidth,      % cannot be copied easily
    stringstyle=\color{Purple},
    basicstyle=\footnotesize,
]
% We probabilistically un-mix singular values sampled from the nxn GUE
% to produce samples distributed as the union of two Laguerre ensembles.

% t = number of samples, n = order of GUE
t = 100000; n = 7; outlist = zeros(t,n);

% We need a list of all partitions of {1,2,...,n} into two balanced parts.
% We generate all subsets of size floor(n/2) and their complements.

cbinom = nchoosek(n,floor(n/2));      % - a central binomial coefficient
parta = nchoosek([1:n],floor(n/2));   % - subsets of {1,2,...,n}
partb = zeros(cbinom,ceil(n/2));      % - their complements
for prep = 1:cbinom
    partb(prep,:) = setdiff([1:n],parta(prep,:));
end
partitions = [parta partb];
P = zeros(1,cbinom);

for rep = 1:t;
    % Sample singular values from the GUE of appropriate size
    G = randn(n)+i*randn(n); A = (G+G')/2; eiglist = sort(abs(eig(A)));

    % We'll need the differences of the squares of the eigenvalues
    singdiffs = (eiglist.^2*ones(1,n) - (eiglist.^2*ones(1,n))');

    % The n eigenvalues can be partitioned in binom(n,floor(n/2)) ways.
    % Compute the relative densities with common factors ommited. 
    for prep = 1:cbinom
	P(prep) = (abs(prod(eiglist(parta(prep,:))))/prod(prod(singdiffs(parta(prep,:),partb(prep,:)))))^2;
    end;

    % Separate the singular values with each partition occuring proportionally to its density
    outlist(rep,:) = eiglist(partitions(find(cumsum(P)>rand*sum(P),1),:))';
end

figure(1); hist(reshape(outlist(:,1:3),1,[]),100);
figure(2); hist(reshape(outlist(:,4:7),1,[]),100);$\mskip200mu\raisebox{3.2in}[0pt][0pt]{\begin{minipage}{2in}\includegraphics[width=2in]{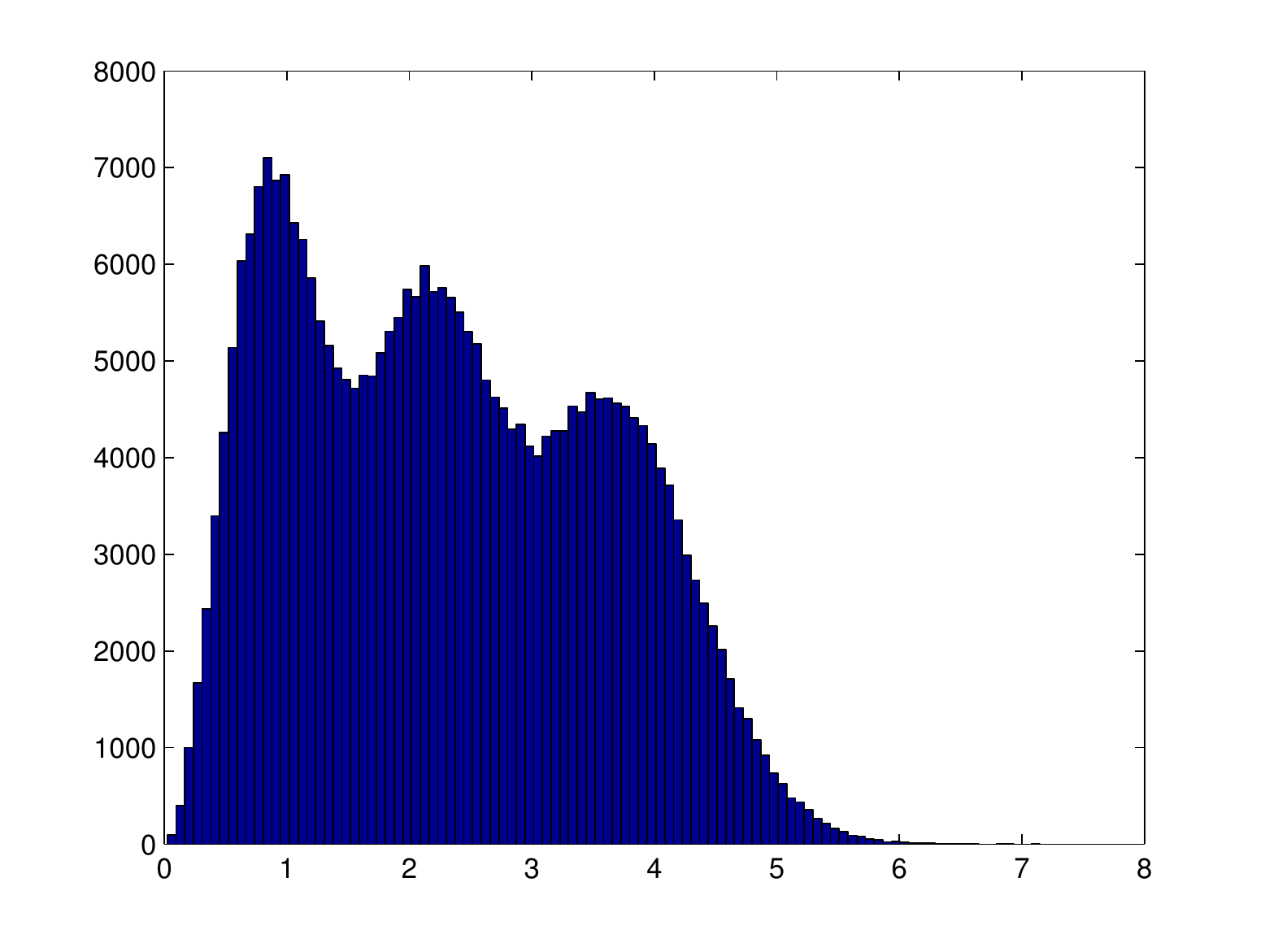}\\[-2.5ex]~Figure~M1\\[3ex]\includegraphics[width=2in]{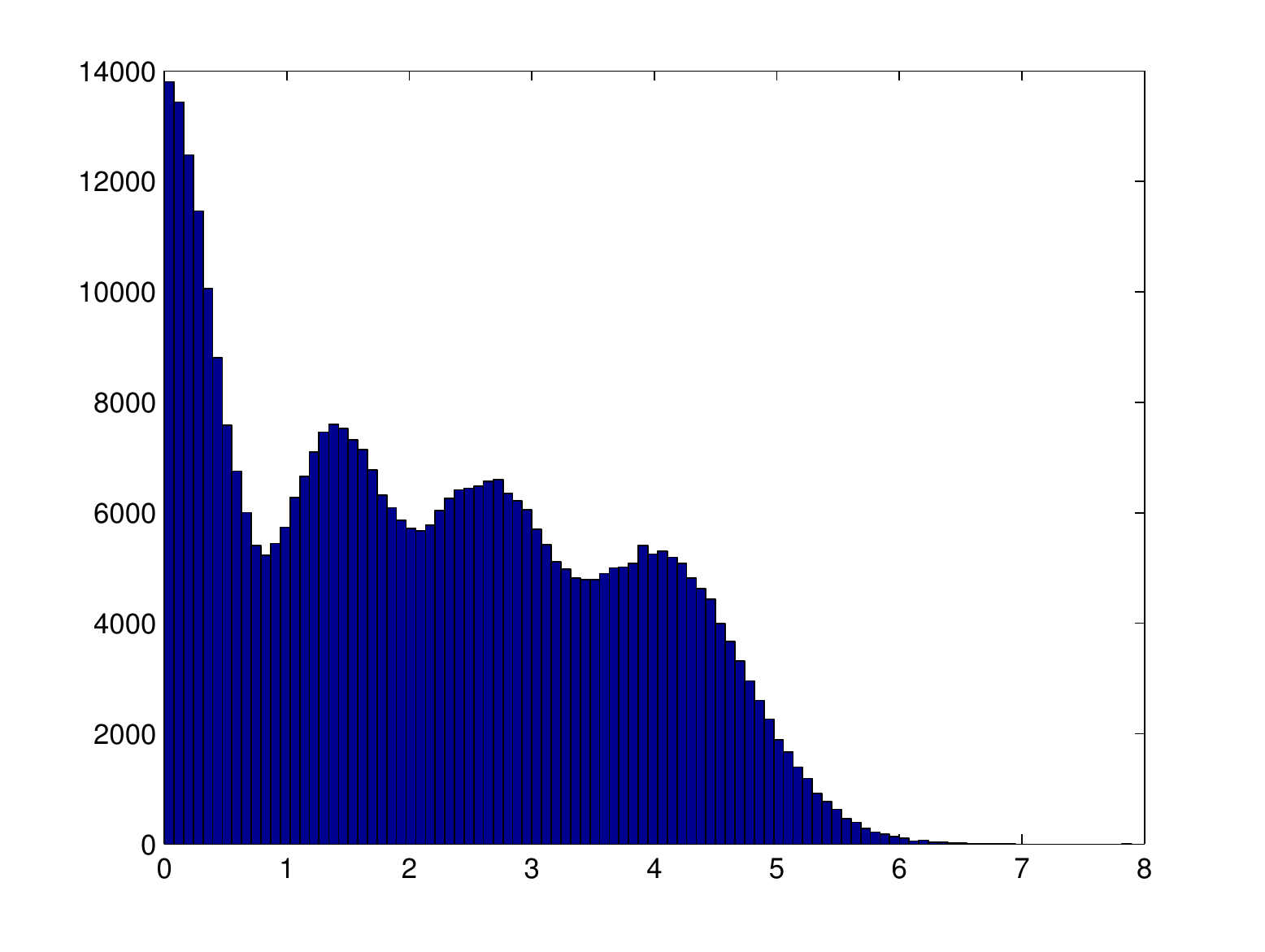}\\[-2.5ex]~Figure~M2\end{minipage}}$
\end{lstlisting}

It appears to be an interesting algorithmic question in general to
determine when a collection of data can be partitioned into two (or
more) independent sets.  In the particular case of the \GUE{}, we know
the exact distributions of the factors, and can thus use conditional
probabilities to generate two independent sets from their union.
Listing~\ref{lst:UnmixingSingularValues} provides a sample of MATLAB
code that demonstrates this.  After a precomputation of the matrices
\texttt{parta}, \texttt{partb}, and \texttt{partitions} holding a list
of appropriately sized subsets and their complements taken from
$\{1,2,\dotsc,n\}$, the main loop generates $t$ samples of the
singular values of $\GUE_n$.  For each sample, relative probabilities
are computed that this data arose from each of the
$\binom{n}{\floor{n/2}}$ partitions of the values between
anti-$\GUE_n$ and anti-$\GUE_{n+1}$ factors.  Choosing a random
partition with weight proportional to these probabilities allows us to
partition the values into the columns of \texttt{outlist}, such that
the first $\floor{n/2}$ columns are distributed as the positive
singular values of an anti-$\GUE_{n}$ and the final $\ceil{n/2}$ are
independently distributed as the positive singular values of an
anti-$\GUE_{n+1}$.  The embedded figures, M1 and M2 show histograms
approximating the level densities for the parts when $n=7$, and should
be compared with the theoretical level densities of
$\LUE_{3}^{(+1/2)}$ and $\LUE_{4}^{(-1/2)}$ from
Figure~\ref{fig:LagPlots}.

\goodbreak

\section{Applications}\label{sec:applications}

The structure exhibited by Theorem~\ref{thm:mainresult} provides new
explanations for existing observations about the distribution of the
determinant and extreme singular values of the \GUE{}.  We do not
believe they have previously been recognized as consequences of the
same underlying structure.

\subsection{The Determinant of the \GUE{}}\label{sec:determinant}

While the determinant of the \GUE{} depends on the signs of
eigenvalues, its absolute value does not.  The bi-diagonal model for
the singular values of the \GUE{}, given in \eqref{eqn:GUEbidiag}, and
illustrated in Figure~\ref{fig:staircasedet} permits us to write it as
a product of independent $\chi$-distributed random variables.  We
obtain directly the expected values of even powers of the determinants
of \GUE{} matrices, and by invoking duality between $k$ and $n$ when
computing the expected value of $\det{M^k}$ for $M\sim\GUE_n$ we
obtain expected values of odd powers as well.  A second consequence is
a direct explanation for the asymptotic log-normality of the absolute
value of large \GUE{} matrices, which had been previously concluded
\emph{via} technical computations in~\cite{D-LC} and~\cite{TV}.

\begin{figure}[!bp]
\includegraphics[width=\textwidth]{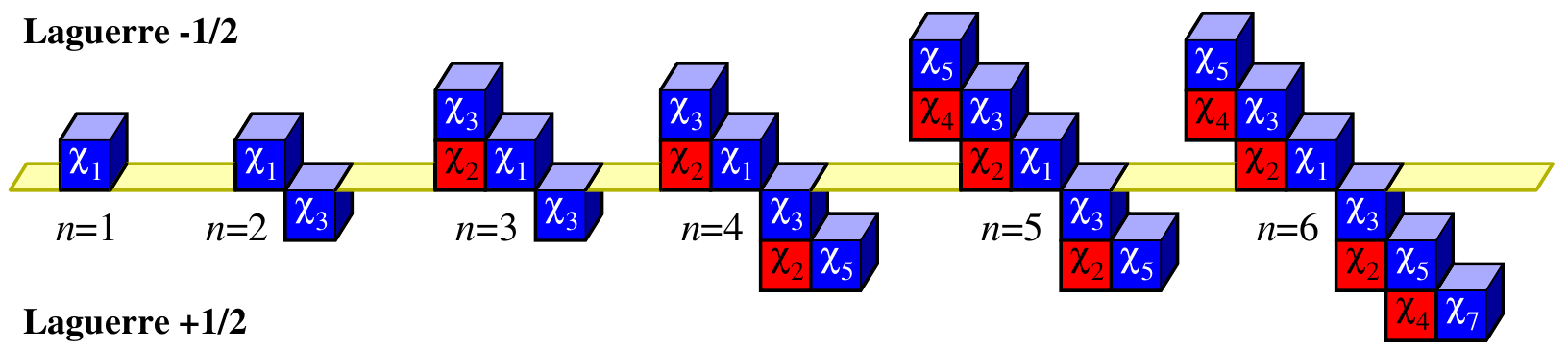}
\caption{Bi-diagonal models for the singular values of the $\GUE{}_n$.
  Only diagonal entries (white on \textcolor{Blue}{blue}) contribute
  to the absolute value of the determinant.}\label{fig:staircasedet}
\end{figure}

The distribution of the determinant of the \GUE{} was previously
studied by Mehta and Normand \cite{MeNo}, who examined the Mellin
transform of its even and odd parts, using sparsity to write each as a
product of determinants.  The distribution of the absolute value of
the determinant corresponds to the transform of the even part, and
using our decomposition we can quickly re-derive this part of their
conclusion: although Mehta and Normand did not interpret the factors
probabilistically, they are the Mellin transforms of
$\chi$-distributed random variables, corresponding to the diagonal
entries of our bi-diagonal model for the singular values of the
\GUE{}.  We do not presently have a corresponding decomposition to
describe the Mellin transform of the odd part of the pdf, but odd
moments of the determinant of $\GUE{}_n$ are either zero for reasons
of symmetry when $n$ is odd, or are recoverable using a duality that
reverses the roles of the order of a matrix and the power of its
determinant.

To describe the distribution of the absolute value (or even powers) of
the determinant of the \GUE{}, it is sufficient to describe the
distribution of the products of the singular values.  Using the
bi-diagonal representation, \eqref{eqn:GUEbidiag}, the product of the
singular values is identified with the product of the independently
$\chi$-distributed diagonal entries of these matrices (see
Figure~\ref{fig:staircasedet}).  The following theorem is an immediate
consequence.

\begin{thm}\label{thm:absdetGUE}
  For $M\sim\GUE_n$, the absolute value of the determinant of $M$ is
  distributed as the product of independent random variables
  $\prod_{i=1}^n X_i$ where $X_i\sim\chi_{2\floor[ ]{\frac{i}{2}}+1}$.
\end{thm}

From Theorem~\ref{thm:absdetGUE} we obtain immediate expressions for
expected values of even moments of the determinant in terms of the
moments of $\chi$-distributed random variables, which are not
altered by considering products of singular values instead of
eigenvalues.  Symmetry implies that odd moments of determinants of odd
order $\GUE$ vanish, and it remains only to determine the expected
values of odd moments of \GUE{} of even orders.  For this we invoke
the following Lemma, which is a special case of a duality principal
described by Dumitriu in \cite[Theorem~8.5.3]{Du-Thesis}, where a more
general result was derived using the machinery of symmetric function
theory (see \cite{Macdonald, Stanley-Jack}), based on the observation
that powers of the determinant can be expressed in terms of
evaluations of Jack symmetric functions.

\begin{lemma}[{Dumitriu \cite[part of Theorem~8.5.3]{Du-Thesis}}]\label{lem:DumitriuDuality} If
  $n$ and $k$ are positive integers, at least one of which is even, then
  \[
  \E{\det(M_n^k)}=(-1)^{nk/2}\E{\det(M_k^n)},
  \]
  where $M_n\sim\GUE_n$ and $M_k\sim\GUE_k$.
\end{lemma}

We thus obtain division-free expressions for the moments of the
determinant of the \GUE{}.  These moments were previously given by
Mehta and Normand (for a different choice of normalization) in
\cite{MeNo} as products of ratios of $\Gamma$-functions, and with a
more direct derivation by Andrews~\emph{et al.} in
\cite[Theorem~1]{AGJ} as products of ratios of factorials.  In both
cases, the moments were listed in four cases, depending on the
parities of both $n$ and $k$.

\begin{cor}\label{cor:detmoments}
  The expected values of powers of the determinant of the $M\sim\GUE_n$ are
  given by:
  \begin{equation}
    \E{\det{M^k}}=\left\{\begin{array}{cl}
    \displaystyle\prod_{i=1}^{n}\prod_{j=1}^{k/2}\Big(2\floor[\Big]{\frac{i}{2}}+2j-1\Big)
    &\text{if $k$ is even}\\
    (-1)^{nk/2}\displaystyle\prod_{i=1}^{k}\prod_{j=1}^{n/2}\Big(2\floor[\Big]{\frac{i}{2}}+2j-1\Big)
    &\text{if $n$ is even}\\
    0&\text{if $n$ and $k$ are both odd.}
    \end{array}\right.
  \end{equation}
\end{cor}

\begin{table}
\caption{For $M\sim\GUE_n$, we write $\E{\det(M^k)}$ as
  a product of odd integers.}\label{tab:detmoments}\vskip-2ex
\hspace*{\fill}\includegraphics[]{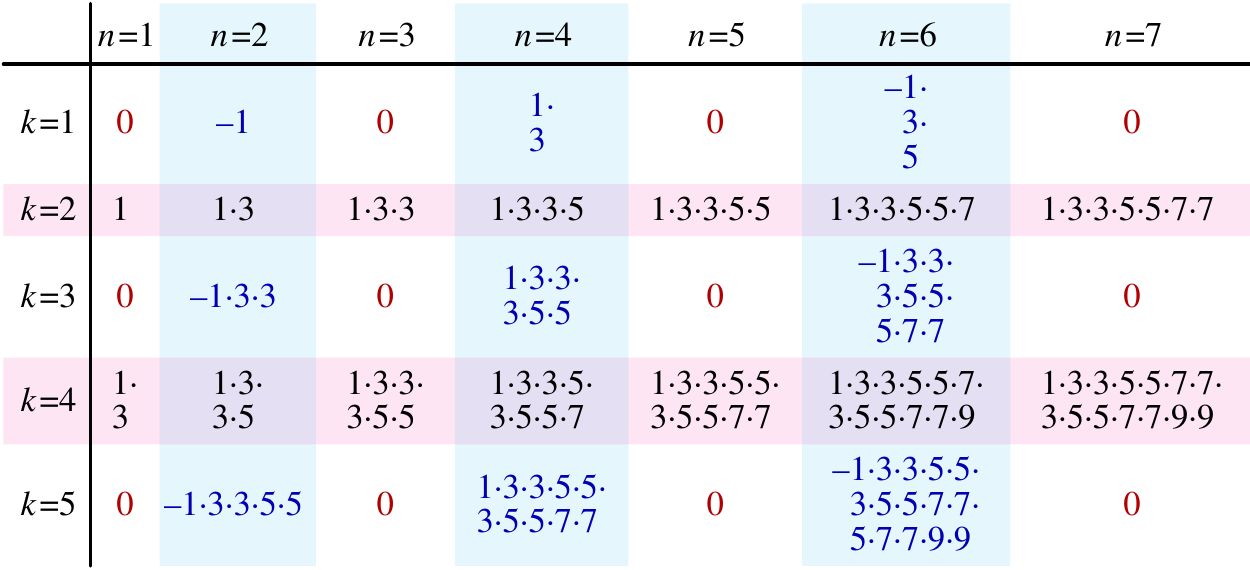}\hspace{\fill}\mbox{}
\end{table}

\begin{rem}
  When $n$ and $k$ are both even, Corollary~\ref{cor:detmoments}
  provides two valid formulae.  For $M\sim\GUE_6$ the first
  formula gives
  $\E{\det{M^4}}=(1\cdot3\cdot3\cdot5\cdot5\cdot7)(3\cdot5\cdot5\cdot7\cdot7\cdot9)$,
  which interpolates between $n=5$ and $n=7$ entries of the $k=4$ row
  of Table~\ref{tab:detmoments}, while the second formula presents the
  factorization in the form
  $(1\cdot3\cdot3\cdot5)(3\cdot5\cdot5\cdot7)(5\cdot7\cdot7\cdot9)$,
  which interpolates between $k=3$ and $k=5$ entries of the $n=6$
  column of the same table.
\end{rem}

\begin{rem}
  A consequence of Corollary~\ref{cor:detmoments} is that the expected
  value of the $2k$--th power of the determinant of a matrix sampled
  from the $n\times{}n$ \GUE{} is a product of odd integers, none of
  which exceeds $2\floor{\frac{n}{2}}+2k-1$.  These factorizations are
  given explicitly in Table~\ref{tab:detmoments}.  This can be
  considered a signature of the fact that the singular values of the
  \GUE{} have a bi-diagonal model with independent $\chi$-distributed
  diagonal entries.  By contrast, moments of the determinants of
  \GOE{} and \GSE{} matrices of even order can involve
  larger prime factors.  For example,
  $\E[\GOE_2]{\det(M^6)}=(3^2)(5^2)(167)$ and
  $\E[\GSE_4]{\det(M^6)}=(3^6)(5^2)(7^2)(11)(347)$, with appropriate scaling.
\end{rem}

In \cite{D-LC}, Delannay and Le~Ca{\"e}r used the Mellin transform of
the distribution of the determinant of the \GUE{} from \cite{MeNo} to
establish the asymptotic log-normality of its absolute value, and
presented analogous computations for the \GOE{}.  Tao and Vu gave new
parallel derivations of the distributions of $\log\abs{\det(A)}$ when
$A$ is distributed as $\GUE_n$ or $\GOE_n$, and proved asymptotic
normality in both cases (\cite[Theorem~4]{TV}).  They speculated that
normality cannot be explained as a consequence of the existence of
underlying independent random variables.  In fact the logarithms of
the diagonal entries of our bi-diagonal model provide such variables.
This is exactly analogous to the case of Wishart matrices, analyzed by
Goodman in \cite{Goodman}, but here the $\chi$-distributed factors
occur only with odd degrees of freedom.  Properties of logarithms of
$\chi$-distributed random variables give the \GUE{} half of
\cite[Theorem~4]{TV} as a corollary.

\begin{cor}[{Tao and Vu \cite[part of
    Theorem~4]{TV}}]\label{cor:GUElogCLT}
  With $M_n\sim\GUE_n$ we have the central limit theorem,
  \vskip-5ex
  \[
  \frac{\log\abs{\det{M_n}}-\frac{1}{2}\log{n!}+\frac{1}{4}\log{}n}{\sqrt{\frac{1}{2}\log{}n}}
  \mathrel{\overset{\mathrm{d}}{\to}}N(0,1)
  \]
  where $\overset{\mathrm{d}}{\to}$ denotes
  convergence in distribution.
\end{cor}
\begin{proof}[Proof sketch]
    From Theorem~\ref{thm:absdetGUE}, if $M_n\sim\GUE_n$, then
    $\log\abs{\det(M_n)}\sim\sum_{i=1}^n\log\abs{X_i}$, is a sum of $n$
    independent random variables with $X_i\sim\chi_{2\floor[ ]{\frac{i}{2}}+1}$.
    Now the expected value of the logarithm of a $\chi_k$-distributed
    random variable, $X$, is given by
    \[  
    \mu_k\coloneqq\E{\log(X)}=\frac{2^{1-k/2}}{\Gamma(k/2)}
    \int_0^\infty \log(x) x^{k-1}\mathrm{e}^{-x^2/2}\;\mathrm{d}x =
    \frac{1}{2}\psi\Big(\frac{k}{2}\Big)+\frac{1}{2}\log2
    \]
    where $\psi(x)\coloneqq\frac{\mathrm{d}}{\mathrm{d}x}\ln\Gamma(x)=\frac{\Gamma'(x)}{\Gamma(x)}$ is
    the digamma function, which satisfies the summation formula
    \[
      \sum_{l=0}^{n-1}\psi\Big(l+\frac{1}{2}\Big)
      =\Big(n-\frac{1}{2}\Big)\psi\Big(n+\frac12\Big)-\frac{\gamma}{2}-\log2-n.
    \]
    Applying the summation to the even and odd terms of our expression
    for $\log\abs{\det(M_n)}$ we conclude 
    \[
    \E{\log\abs{\det{}M_n}}=\frac{n}{2}\psi
    \Big(\ceil[\Big]{\frac{n}{2}}+\frac{1}{2}\Big)
    +\frac{n}{2}\log{2}-\ceil[\Big]{\frac{n}{2}}=\frac{1}{2}\log\frac{n!}{\sqrt{2\pi{}n}}+O(n^{-1}),
    \]
    and hence that
    \[
    \lim_{n\to\infty}\E{\log\abs{\det{M_n}}}
    -\frac{1}{2}\log{}n!+\frac{1}{4}\log{}n=-\frac{1}{4}\log2-\frac{1}{4}\log{\pi}
    =-0.459469\dots
    .
    \]
    Similarly, the variance of the logarithm of a $\chi_k$-distributed
    random variable, $X$, is given by
    \[
    \sigma_k^2\coloneqq\E{(\log{}X-\mu_k)^2}=\frac{1}{4}\psi_1\Big(\frac{k}{2}\Big)
    ,
    \]
    where
    $\psi_1(x)\coloneqq\frac{\mathrm{d}}{\mathrm{d}x}\psi(x)$ is
    the trigamma function, which satisfies the recurrence
    \[
    \sum_{l=0}^{n-1}\psi_1\Big(l+\frac12\Big)=\Big(n-\frac12\Big)\psi_1\Big(n+\frac12\Big)
    +\psi\Big(n+\frac{1}{2}\Big)+\gamma+\frac{\pi^2}{4}+2\log2.
    \]
    Again summing over even and odd terms, we get
    \begin{align*}
    \Var\big[\log\abs{\det{}M_n}\big]&=
    \frac{n}{4}\psi_1\Big(\ceil[\Big]{\frac{n}{2}}+\frac{1}{2}\Big)
    +\frac{1}{2}\psi\Big(\ceil[\Big]{\frac{n}{2}}+\frac{1}{2}\Big)+\frac{\gamma}{2}+\log2\\
    &=\frac{1}{2}\log{}n+\frac{1}{2}(\gamma+1+\log2)+O(n^{-2}),
    \end{align*}
    and obtain the limit
    \[
    \lim_{n\to\infty}\Var\big[\log\abs{\det{M_n}}\big]-\frac{1}{2}\log{n}=
    \frac{1}{2}(\gamma+1+\log2)=1.1351814\dots.
    \]
    Asymptotic normality, and the stated central limit theorom, follow
    by checking the Lyapunov condition for fourth moments.  Letting
    $\beta_k$ denote the fourth central moment of a $\chi_k$
    distributed random variable, $X$, we have
    \[
    \beta_k=\E{(\log
      X-\mu_k)^4}=\frac{3}{16}\psi_1\Big(\frac{k}{2}\Big)^2+\frac{1}{16}\psi_3\Big(\frac{k}{2}\Big).
    \]
    Since the polygamma function $\psi_3(x)=\psi'''(x)$ satisfies
    the bound $\psi_3(x)\leq\psi_1(x)\leq1$ when $x\geq2$, we conclude
    that with $s_n^2=\Var\big[\log\abs{\det{M_n}}\big]$
    \[
    \sum_{l=1}^n\beta_{2\floor[ ]{\frac{l}{2}}+1}=O(s_n^2)
    \qquad\text{so}\qquad
    \lim_{n\to\infty}\frac{1}{s_n^4}\sum_{l=1}^{n}\beta_{2\floor{\frac{l}{2}+1}}=0.\qedhere
    \]
\end{proof}


Andrews \emph{et al.} also provide relatively compact expressions for
the determinant of the \GOE{} for even order matrices as
\cite[Eq~(23)]{AGJ} and odd order matrices as \cite[Eq~(24)]{AGJ},
with less compact forms derived by Delannay and Le~Ca{\"e}r in
\cite{D-LC} and summarized by Mehta as \cite[Eq.~(26.5.11),
Eq.~(26.6.15), and Eq.~(26.6.16)]{Mehta}.  After observing numerically
that the moments for odd order matrices have only small prime factors,
we identified the following theorem.


\begin{thm}\label{thm:oddGOEdet}
  For $M\sim\GOE_{2n+1}$, the determinant of $M$ has the same moments
  as the product of independent random variables
  $\sqrt{2}\,X\prod_{i=1}^{n} Y_i$ where $X\sim{}N(0,1)$ and
  $Y_{i}\sim\chi_{2i+1}^{2}$.
\end{thm}
\begin{proof}
  The distributions of the determinant and our purported product
  are both symmetric about zero, so it is sufficient to show that even
  moments agree.  Beginning with 
  \cite[Eq.~(24)]{AGJ},
  \[
  \E[\GOE_{2n+1}]{\det{M^{2u}}}=2^u(2u-1)!!\prod_{j=1}^{2u}\frac{(2n+2j-1)!!}{(2j-1)!!},
  \]
  we recognize $(2u-1)!!$ as the expected value of $X^{2u}$ for
  $X\sim{}N(0,1)$.  By rewriting the product as
  \[
  \prod_{j=1}^{2u}\frac{(2n+2j-1)!!}{(2j-1)!!}=\prod_{i=1}^{n}\bigg(\prod_{k=0}^{2u-1}(2i+1+2k)\bigg)
  \]
  we obtain a division-free expression for the moments of the
  determinant, and observe that the factors are the expected values of
  $Y_i^{2u}$ for $Y_i\sim\chi_{2i+1}^2$.
\end{proof}

\begin{rem}
  In \cite{Bornemann-LaCroix-GOE}, Bornemann and La$\;$Croix give a
  direct interpretation for the factors in
  Theorem~\ref{thm:oddGOEdet}.  A corresponding description of the
  distribution of the modulus of the determinant for even order
  matrices then allowed them to derive the central limit
  theorem for the \GOE{} analog to Corollary~\ref{cor:GUElogCLT}.
\end{rem}


\subsection{Extreme Singular Values and the Condition Number of the
  \GUE{}}\label{sec:extreme}

We use the condition number of a matrix to motivate the study of the
distributions of the largest and smallest singular values of the
\GUE{}.  From Theorem~\ref{thm:mainresult}, these are related to
singular values of Laguerre ensembles with parameters $\pm1/2$.  This
perspective unifies some 
asymptotic results and also
suggests the need for additional special functions for describing
the singular value of Laguerre ensembles with
non-integer parameters.  Of particular note, the smallest singular
value of the \GUE{}, which is associated with the bulk scaling limit,
is described in terms of the smallest singular values of Laguerre
ensembles, which are associated with hard-edge limits, giving the
\GUE{} a sort of virtual hard edge.

The \emph{condition number} of a matrix predicts stability in
numerical linear algebra, and is given by the ratio of its largest and
smallest singular values.  In practice, fluctuations of the largest
singular value are small, and the distribution of the condition number
can be approximated by considering only the smallest singular value
(see the first author's analysis of the corresponding problem for
Laguerre ensembles in \cite{Edelman-condition}).  For the purpose of
analyzing the condition number, the signs of eigenvalues introduce
noise which we can ignore by partitioning the singular values of the
\GUE{} according to Theorem~\ref{thm:mainresult}.  We can then analyze
both the smallest and largest singular values as functions of
independent quantities.  This makes the product structures of the
eigenvalue counting functions in the limits into extensions of
corresponding finite factorizations.

The earliest results we have identified along these lines involve gap
probabilities in the bulk scaling limit of the \GUE{}.  The
probability that the smallest singular value is at least $s$ is also
the probability that there are no eigenvalues between $-s$ and $s$.
In the bulk-scaling limit, this is known as the \emph{gap
  probability}, and, by translation invariance, becomes independent of
the particular interval chosen.  A more general problem is to describe
the eigenvalue counting function, $E_2(k;s)$, giving the probability
that a random interval of length $2s$ contains precisely $k$
eigenvalues.  The corresponding problem for the \GOE{} was given a
Fredholm determinantal representation by Gaudin in
\cite{Gaudin1961447}.  This was adapted to the \CUE{} by Dyson in
\cite{Dyson-STELCSIII}, and extended by Mehta and des~Cloizeaux in
\cite{MehClo} to encompass $k\neq0$.  Bornemann provides an excellent
summary of the computational implications of this and related results
in \cite{Bornemann} where the following appears as his equation~(5.7)
\begin{equation}\label{eqn:bulksum}
  E_2(k;s)=\sum_{j=0}^{k}E_+(j;s)E_-(k-j;s). 
\end{equation}
Note that $E_+$ and $E_-$ are not themselves eigenvalue counting
functions, but defined instead in terms of the decomposition of the
sine kernel by its orthogonal actions on even and odd functions.  In
fact, the right side of \eqref{eqn:bulksum} can also be interpreted as
a limit of counting functions for singular values.  When we showed
this to Bornemann, he observed that the finite version can be obtained
using kernel methods parallel to the derivation of the limiting case,
and that this is essentially the content of Forrester's observations
about the \GUE{} and related ensembles in \cite{Forrester-Evenness}.

Most of the required observations are already present in
\cite[Ch.~6]{Mehta}, where Mehta notes the characteristic checkerboard
sparsity pattern as the reason that the determinantal representation
of $E_2(0;s)$ factors.  This factorization is not a consequence of the
relationship between the bulk-scaling limits of the \GUE{} and \CUE{}
(such a factorization is somewhat less surprising for the \CUE{},
because of properties of compact Lie groups, as explored by Rains
\cite{Rains-Powers, Rains-power-images}).  In fact a similar
factorization occurs for \GUE{} of finite order, and we recognize the
factors as corresponding to the complementary cdfs of the smallest
singular values of the two component anti-\GUE{}s implicit in
Theorem~\ref{thm:mainresult}.

An initial analysis of extreme singular values is harmed by the
inclination to partition eigenvalues according to sign.  The largest
singular value is then a function of the largest and smallest
eigenvalues, while the smallest singular values is determined by the
least positive and greatest negative eigenvalues.  For asymptotically
large \GUE{} matrices, the largest and smallest eigenvalues are at
soft-edges, and their distributions are described by the Tracy-Widom
law.  In the large $n$ limit, these two soft-edges become independent,
and the cumulative distribution function for the largest singular
value factors as the product of cumulative distribution functions for
the two edges.  Our intuition suggests that for moderately-sized
matrices, the largest and smallest eigenvalues are essentially
independent, so we expect a near factorization for large but finite
$n$.  To interpret \eqref{eqn:bulksum} as a limit we need to change
perspectives.  Instead of two identical distributions (those of the
largest and smallest eigenvalues) becoming independent, the finite
version involves two independent distributions becoming identical (the
soft edges of $\LUE^{(+1/2)}$ and $\LUE^{(-1/2)}$).

By contrast, the smallest singular value is either the least
positive eigenvalue, or the greatest negative eigenvalue, and these
two quantities do not become independent for large matrices.  This
makes the factorization of the complementary cdf of the smallest
singular value much more surprising.


The symmetry between smallest and largest singular values is most
easily phrased in terms of singular value counting functions.  We
denote these by $S$, and define them as analogs to the $E$ discussed
previously.  For a subset $J\subseteq\mathbb{R}_+$, we let
$S_2^{(n)}(k,J)$ denote the probability that exactly $k$ singular
values of $\GUE_n$ lie in the set $J$.  Similarly, we let
$S_{\LUE_{\pm}}^{(n)}(k;J)$ denote the corresponding probabilities
that exactly $k$ singular values of the bi-diagonal matrix modeling
the $\LUE_n^{(\pm1/2)}$, recall \eqref{eqn:bi-diagonalform}, lie in
$J$.  Note that $S_{\LUE_-}$ and $S_{\LUE_+}$ can also be considered
as the even and odd elements of a one-parameter family, since we could
have defined them equivalently by relating $S_{\LUE_-}^{(n)}$ to the
singular values of the order $2n$ anti-$\GUE{}$ and $S_{\LUE_+}^{(n)}$
to the order $2n+1$ anti-$\GUE{}$.  With this notation, the following
corollary to Theorem~\ref{thm:mainresult}, equivalent forms of which
were considered by Forrester in \cite{Forrester-Evenness}, is
immediate.

\begin{cor}\label{cor:singcount} For any measurable subset $J\subseteq\mathbb{R}_+$, the
  singular value counting functions of $\GUE_n$ can be expressed in
  terms of $\LUE$ counting functions \emph{via}
\begin{equation}\label{eqn:singcount}
  S_2^{(n)}(k;J)=\sum_{j=0}^k S_{\LUE_{+}}^{(n_1)}(j;J)
  S_{\LUE_{-}}^{(n_2)}(k-j;J),
\end{equation}
where $n_1=\floor[ ]{\frac{n}{2}}$ and $n_2=\ceil[ ]{\frac{n}{2}}$.
\end{cor}

\begin{rem}
  The similarity between \eqref{eqn:bulksum} and \eqref{eqn:singcount}
  is intentional, but also slightly overstated.  On passing to limits,
  the r\^oles of the subscripts `$+$' and `$-$' are reversed.  The
  Laguerre $+1/2$ ensembles are associated with the action of the sine
  kernel on odd functions, which correspond to the $E_-$ factors in
  \eqref{eqn:bulksum}.  Similarly, on passing to
  the limit, the $S_-$ factors in \eqref{eqn:singcount} correspond to
  $E_+$ factors in \eqref{eqn:bulksum}.
\end{rem}

Forrester summarizes these results into a generating series identity
encompassing all $k$ simultaneously.  We take the opposite emphasis,
and believe that via their relationship to condition numbers, the two
particular cases $k=0$ and $k=n$ are of special interest.  In these
cases, the sum consists of a single non-zero term, so the result takes
the form of a factorization.

%

\begin{exam}
  Suppose that $X$, $Y$, and $Z$ are the minimum singular value of
  matrices sampled from $\GUE_7$, and nominal $3\times3.5$ and
  $4\times3.5$ matrices (corresponding to $\LUE_3^{(+1/2)}$ and
  $\LUE_4^{-1/2}$ respectively).  The cumulative distribution function for $X$,
  $Y$, and $Z$ are plotted in Figure~\ref{fig:minimumsing} (top).  By
  specializing \eqref{eqn:singcount} to $n=7$ and $k=1$, we find that
  \[
  1-\mathrm{P}(X\leq{}x)=\big(1-\mathrm{P}(Y\leq{}x)\big)
                         \big(1-\mathrm{P}(Z\leq{x})\big).
  \]
  To observe this graphically, we take logarithms, and note that
  $\log\big(1-\mathrm{P}(X\leq{}x)\big)$ is the average of
  $2\log\big(1-\mathrm{P}(Y\leq{}x)\big)$ and
  $2\log\big(1-\mathrm{P}(Z\leq{x})\big)$, a relationship plotted in
  Figure~\ref{fig:minimumsing} (bottom).
\end{exam}



\begin{figure}
\hspace*{\fill}\includegraphics[width=.9\textwidth]{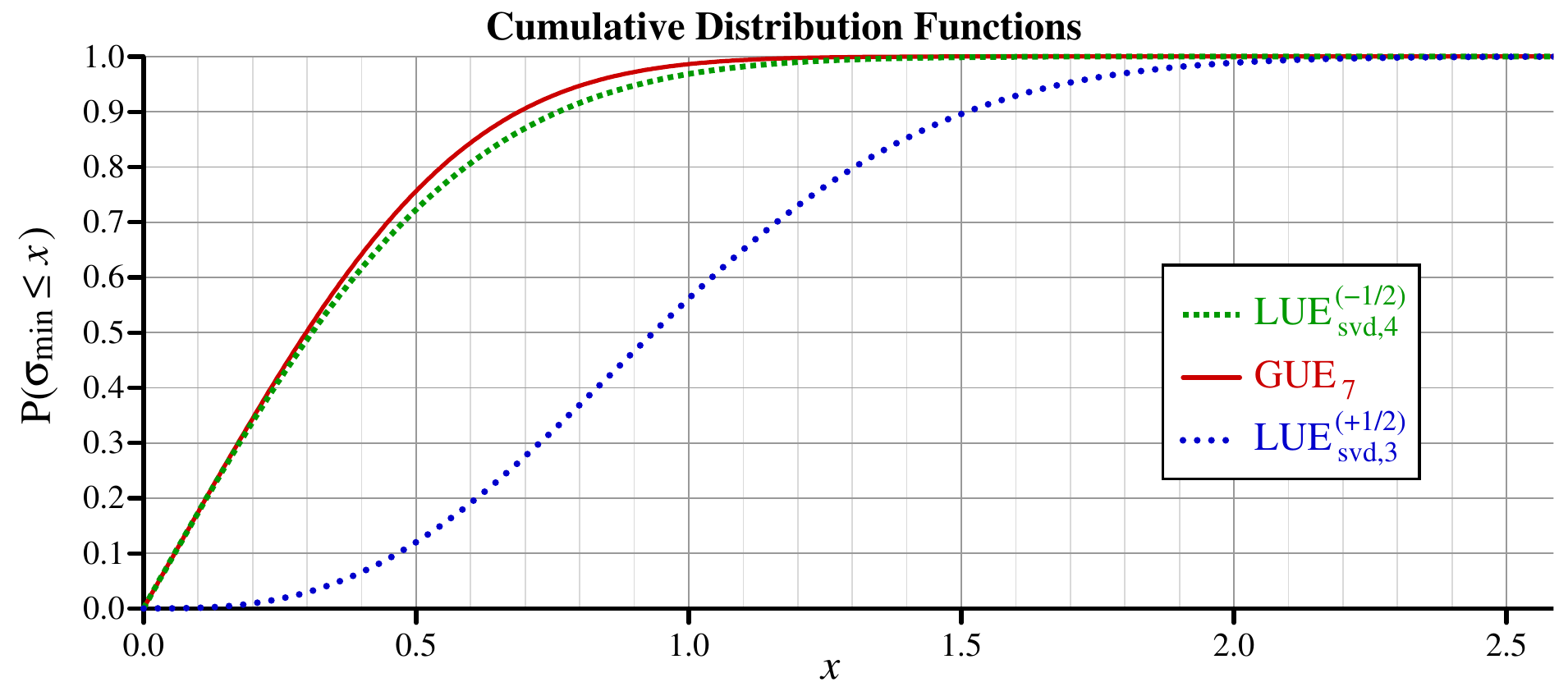}\hspace{\fill}\mbox{}
%
%
\hspace*{\fill}\includegraphics[width=.9\textwidth]{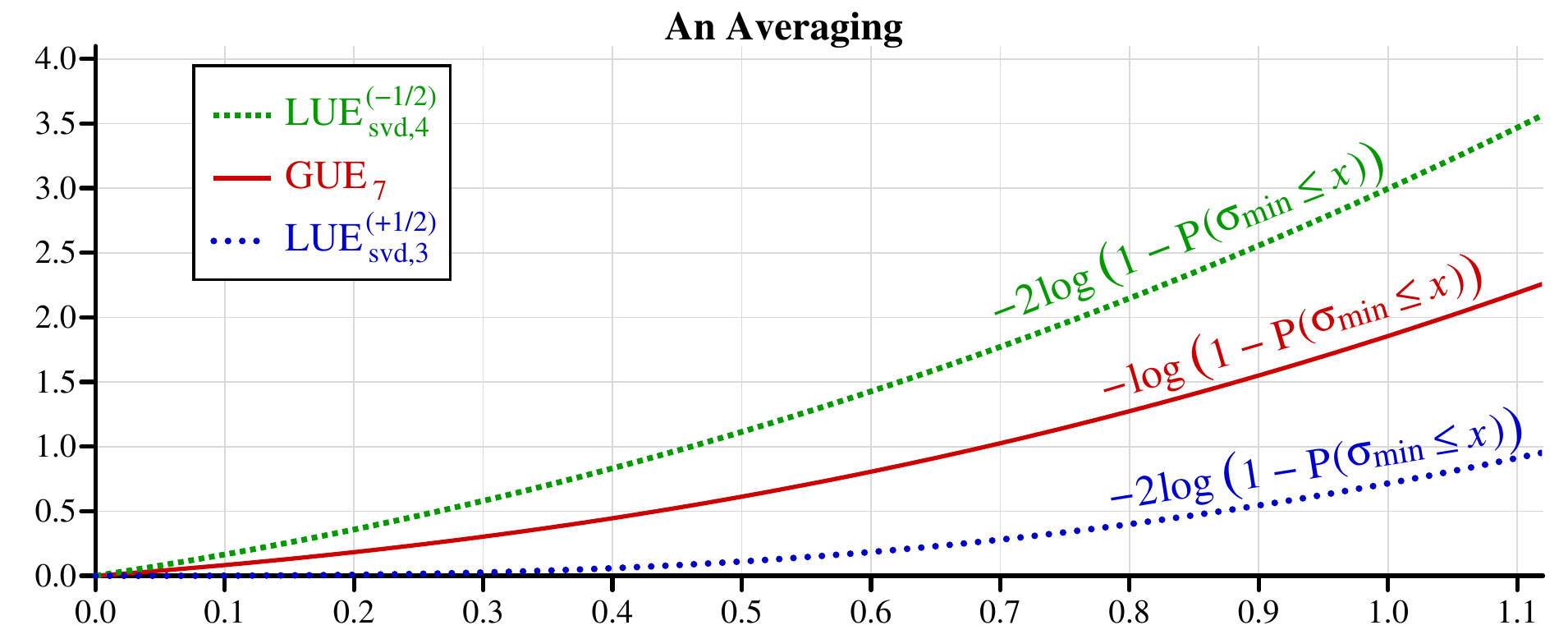}\hspace{\fill}\mbox{}
\caption{The cdfs for $\sigma_{\min}$ of \textcolor{BrickRed}{$\GUE{}_7$},
  \textcolor{OliveGreen}{$\LUE_{\mathrm{svd},4}^{(-1/2)}$}, and
  \textcolor{blue}{$\LUE_{\mathrm{svd},3}^{(+1/2)}$} (top) are related
  via the logarithms of their complements (bottom).
}\label{fig:minimumsing}
\end{figure}

One of the initial aims of our investigation was to describe the
manner in which the distribution of the smallest singular value of the
$\GUE_n$ depends on $n$.  We had hoped that by specializing
\eqref{eqn:singcount} to $k=0$ we could bootstrap from corresponding
descriptions for Laguerre ensembles.  For integral values of $a$, the
distribution of the smallest singular value of $\LUE_n^{(a)}$ is
described by a confluent hypergeometric function of matrix argument
\cite{Du-Thesis, KoEd}.  The series expansions of these functions are,
however, non-convergent when $a$ is not an integer, and we are unsure
what the appropriate analog is in this situation.  Similarly when $a$
is integral, Forrester and Hughes showed in \cite{FoHu} that relevant
probabilities could be computed as $a\times{}a$ determinants of
matrices with generalized Laguerre polynomials as entries.  This
allows a natural generalization to non-integer $n$, but only when $a$
is an integer.  Similarly the recurrence they identified from the
double Wronskian structure of the determinants in not closed in this
case.

The distributions can still be described for all $a$ as determinants
involving special functions.  In particular, a direct computation
gives the probability that there are no singular values less than $s$
in terms of determinants of Hankel matrices of upper incomplete
$\Gamma$-functions.  Letting
\[
F(a,n,s)\coloneqq\det\big(\Gamma(i+j-1+a,s)\big)_{1\leq{}i,j\leq{}n}
\]
where
$\Gamma(s,x)=\int_{x}^{\infty}t^{s-1}\mathrm{e}^{-t}\;\mathrm{d}t$ is
the upper incomplete gamma function, we can integrate the Vandermonde
term by term to see that the complementary cdf for the smallest
singular value of the \GUE{} is
\begin{align*}
  \mathrm{P}_{\GUE_n}(\sigma_{\min}\ge{}s) &= 
  \frac{F\big(-\frac{1}{2},\ceil[ ]{\frac{n}{2}},\frac{s^2}{2}\big)}
       {F\big(-\frac{1}{2},\ceil[ ]{\frac{n}{2}},0\big)} \times 
  \frac{F\big(+\frac{1}{2},\floor[ ]{\frac{n}{2}},\frac{s^2}{2}\big)}
       {F\big(+\frac{1}{2},\floor[ ]{\frac{n}{2}},0\big)}. 
\end{align*}
For small values $n$ these can be evaluated directly, but the matrices
become ill-conditioned as $n$ grows.  These expressions can be
generalized to any $a>1$, but only when $n$ is an integer.

It should be noted that the computation of
$S_2^{(n)}\big(k;(0,s)\big)$ is well-suited to the numerical Fredholm
techniques described by Bornemann in \cite{Bornemann}.  In this
setting, it is unclear under what conditions, if any, it is preferable
to work with the Laguerre factors instead of the \GUE{} expression
directly.

\section{Relationship to Complex Ginibre Ensembles}\label{sec:Ginibre}

We close by noting that Theorem~\ref{thm:mainresult} also has
parallels involving complex Ginibre ensembles (see
\cite[Ch.~15]{Mehta} for a discussion of these ensembles).  A slight
modification of the proof of Theorem~\ref{thm:mainresult} shows that
the magnitudes of the eigenvalues of Ginibre matrices are independent
$\chi$-distributed random variables, each having a different even
number of degrees of freedom.  Furthermore, the phases of the
eigenvalues of sufficiently high powers of such matrices are also
independent.  We speculate that the Ginibre ensembles could be a
bridge to connect our observations to results of Rains involving
powers of compact Lie groups, and to the unitary groups in particular
(see \cite{Rains-Powers, Rains-power-images}), the connection with
which was pointed out to us in discussion with Paul Bourgade.  We
outline some of the relevant properties of Ginibre ensembles, and
sketch some of the reasons that we think they may lie at the center of
the theory.

To give a more concrete motivation for considering Ginibre matrices,
we note that instead decomposing the singular values of the \GUE{}, we
could equally well have decomposed the eigenvalues of the square of
the \GUE{}.  In fact, it is this formulation that most closely matches
the combinatorial identities of Jackson and Visentin \cite{JPV,
  JV-Characters, JV-Eulerian}.  The maps involved are enumerated by
even moments of the \GUE{}, but the relevant maps can also be embedded
injectively into the class of bipartite maps, and when enumerated as
such are naturally related to moments of the complex Ginibre ensemble.
In fact this may be the more natural setting for analyzing map
combinatorics, since the combinatorially related triangulation
conjecture from \cite{JV-Regular, JV-face-colored} can formulated in
terms of Ginibre matrices, but does not appear to have a formulation
in terms of the \GUE{}.

The main result of this section, Theorem~\ref{thm:ginibreindep}, was
previously described by Kostlan in \cite{Kostlan}.  We rederive it
here in a manner intended to emphasize the similarity with
Theorem~\ref{thm:mainresult}.

If no symmetry is imposed on a square matrix with independent complex
Gaussian entries of unit variance (i.e. real and imaginary parts
independently each have variance $\frac12$), then the resulting
eigenvalues are generically complex valued.  The joint eigenvalue
density for such an ensemble was derived by Ginibre, after whom such
matrices are named, in \cite{Ginibre}.  It is supported on
$\mathbb{C}^n$, and is given by
\begin{equation}\label{eqn:ginibredensity}
  p^{\mathrm{G}}_n(z_1,z_2,\dotsc,z_n)=c^{\mathrm{G}}_n {\prod_{1\leq
      i<j\leq{}n}\abs{z_i-z_j}^2} \;\exp\Big(-\frac12\sum_{i=1}^n
  \abs{z_i}^2\Big) \differentials{z}{1}{n},
\end{equation}
which differs from the joint density for the \GUE{} only in its support.
As with Theorem~\ref{thm:mainresult}, we can show that the magnitudes
of the eigenvalues of the Ginibre ensemble are independent, or with a
slight modification that the eigenvalues of large powers of Ginibre
matrices are independent.

As with the \GUE{} it is possible to evaluate expectations of
symmetric functions of the magnitude of eigenvalues by de-symmetrizing
relevant integrals.  In this case, proceeding as with the proof of
Theorem~\ref{thm:mainresult}, and using the fact that
$\abs{z}^2=z\overline{z}$, integrals can be taken against
\[
  n!\,c_n^{\mathrm{G}} \det\begin{pmatrix} 1 & {z^*_2} & {z^*_3}^2 &
  \cdots & {z^*_n}^{n-1} \\ z_1 & z_2{z^*_2} & z_3{z^*_3}^2 & \cdots &
  z_n{z^*_n}^{n-1} \\ z_1^2 & z_2^2{z^*_2} & z_3^2{z^*_3}^2 & \cdots &
  z_n^2{z^*_n}^{n-1} \\ \vdots & \vdots & \vdots & \ddots & \vdots
  \\ z_1^{n-1} & z_2^{n-1}{z^*_2} & z_3^{n-1}{z^*_3}^2 & \cdots &
  z_n^{n-1}{z^*_n}^{n-1}
  \end{pmatrix}
  \;\exp\Big(-\frac12\sum_{i=1}^n z_iz^*_i\Big)
  \differentials{z}{1}{n}.
\]
Switching to polar co-ordinates with the substitutions
$z_j=r_j\mathrm{e}^{\mathrm{i}\theta_j}$ and
$\mathrm{d}z_j=r\,\mathrm{d}r_j\mathrm{d}\theta_j$, each phase
variable occurs only in a single column of the matrix.  Integrating
over the phases annihilates all non-diagonal matrix entries and
shows that we could have integrated against the density
\[
\prod_{i=1}^n
r_i^{2i-1}\mathrm{e}^{-\frac{1}{2}{r_i^2}}\;\mathrm{d}r_i.
\]
By recognizing the factors of this density, we recover a result
previously described by Kostlan.

\begin{thm}[{Kostlan \cite{Kostlan}}]\label{thm:ginibreindep}
  The magnitudes of the eigenvalues of the $n\times{}n$ complex
  Ginibre ensemble have the same distribution as a set of independent
  $\chi$ random variables with $2$, $4$, \dots, $2n$ degrees of
  freedom.
\end{thm}

From this information, and the fact that the ensemble is invariant
under multiplication by a complex phase, we immediately obtain the
known distribution of the determinant of the Ginibre ensemble, and the
distribution of a randomly chosen eigenvalue.  In particular, the pdf
for the magnitude of a randomly chosen eigenvalue is given by the
average of the densities of $\chi$-distributed random variables (see
Figure~\ref{fig:Ginibreleveldensity}), and takes the form
\[
\frac{x}{n}\frac{\Gamma(n,\frac{x^2}{2})}{\Gamma(n)}.
\]
Paralleling Rains' results about the eigenvalues of compact
Lie groups (\cite{Rains-Powers, Rains-power-images}), a slight
extension shows that for an $n\times{}n$ Ginibre matrix, the phases of
the $k$--th power of the eigenvalues are independent and uniformly
distributed for every $k\geq{}n$.  Together these give a weak form of
the circular law.

\begin{figure}
\hspace*{\fill}%
\includegraphics[width=.95\textwidth]{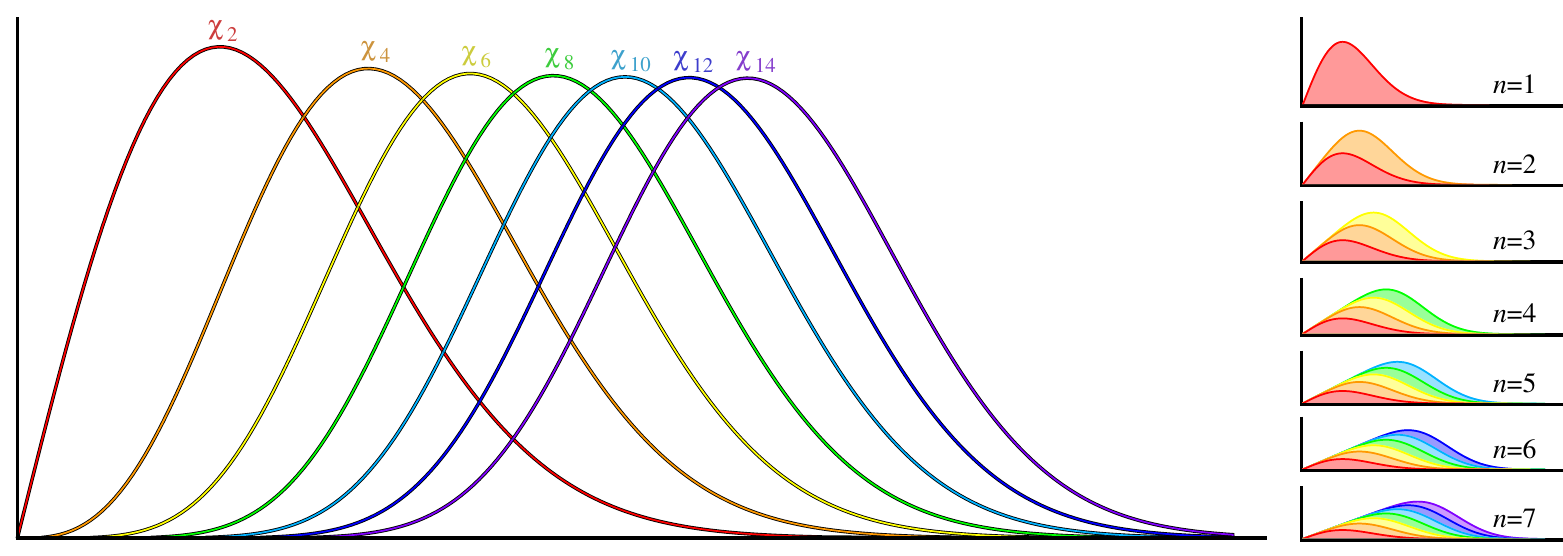}\hspace{\fill}\mbox{}%
\caption{The pdfs of the magnitude of a randomly chosen eigenvalue from
  the $n\times{}n$ complex Ginibre ensemble (right) is the average of
  the pdfs of $\chi$-distributed random variables with consecutive
  even numbers of degrees of freedom (left).}\label{fig:Ginibreleveldensity}
\end{figure}


\section{Related Questions}\label{sec:future}

\begin{itemize}
\item Is there a practical way to identify when a set-valued random
  variable can be generated as a union of independent sets?  In
  particular, given a collection of samples of a set-valued random
  variable can we determine if the same sample distribution can be
  generated as a union of two smaller collections.
\item Are there hidden independencies for the \GOE{} or $\GBE{}$ for
  any $\beta\neq2$?
\item Quantitatively, the absolute value of the determinant of the 
  $n\times n$ \GUE{} can be sampled na\"ively as a function of $n^2$
  independent Gaussian random variables.  Similarly, a $\chi_k$ random
  variable can be sampled as a function of $k$ independent Gaussian
  random variables, so using Theorem~\ref{thm:absdetGUE} we can sample
  the absolute value of the determinant as a function of
  $n^2-\floor{\frac{(n-1)^2}{2}}$ independent Gaussian random
  variables.  This sequence appears in the Online Encyclopedia of
  Integer Sequences as A074148, where a comment suggests that it also
  arises in the context of Cartan decompositions.  To what extent can
  our main theorem be viewed as a relationship between the Unitary
  Group of order $n$ and the Orthogonal Groups of order $n$ and $n+1$?
  To make this question concrete, we have exhibited two ensembles of
  random matrices defined in terms of $n^2$ Gaussian random variables.
  One, the $\GUE_n$, is invariant under conjugation by $U(n)$, while
  the other, anti-$\GUE_{n}\otimes$anti-$\GUE_{n+1}$ is invariant
  under conjugation by $O(n)\times{}O(n+1)$.
\item The decomposition of Theorem~\ref{thm:mainresult} raises the question,
  to what extent can one ensemble be transformed into the other.  It
  would appear that this cannot be accomplished deterministically.
  There is naturally a $2^n\colon\binom{n}{\floor{\frac{n}{2}}}$
  ambiguity.  Any element-wise action would also carry with it an
  implicit mapping between $O(n)\otimes{}O(n+1)$ and $U(n)$.
\end{itemize}

\section*{Acknowledgements}
We would like to thank Folkmar Bornemann, Paul Bourgade, and Peter
Forrester for discussing preliminary versions of our results, and
helping us find new connections with the existing theory.  The second
author would also like to point out that the combinatorial work of
David Jackson and Terry Visentin was essential for inspiring the
present investigation.  We also thank the National Science Foundation for
funding this research through grants DMS--1035400 and  DMS--1016125.

\bibliographystyle{amsplain}
\bibliography{LaCroix}

\end{document}